\documentclass[11pt,reqno]{amsart}
\usepackage[utf8]{inputenc}
\usepackage{amsmath}
\usepackage{amsfonts}
\usepackage{amssymb}
\usepackage{amsthm}
\usepackage{enumitem}

\setlist[description]{leftmargin=\parindent,labelindent=\parindent}

\usepackage{xfrac}
\usepackage[dvipsnames,table]{xcolor}

\usepackage{tikz}
\usepackage[left=2.5cm,right=2.5cm,top=2.5cm,bottom=2.5cm]{geometry}

\usepackage[colorlinks=true,citecolor=MidnightBlue,linkcolor=MidnightBlue,breaklinks=true]{hyperref}
\usepackage[capitalize,nameinlink,noabbrev]{cleveref}
\usepackage[maxbibnames=4,backend=bibtex]{biblatex}
\renewbibmacro{in:}{}
\newtheorem{thm}{Theorem}[section]
\newtheorem{cor}[thm]{Corollary}
\newtheorem{prop}[thm]{Proposition}
\newtheorem{lem}[thm]{Lemma}

\theoremstyle{definition}

\newtheorem{defin}[thm]{Definition}

\theoremstyle{remark}
\newtheorem{obs}[thm]{Remark}

\numberwithin{equation}{section}

\newcommand{\qtens}{\mathbin{\ddot{\otimes}}}
\newcommand{\qf}{\ddot{f}}
\newcommand{\qe}{\ddot{e}}
\newcommand{\qphi}{\ddot{\varphi}}
\newcommand{\qepsilon}{\ddot{\varepsilon}}

\newcommand{\cf}{\widetilde{f}}
\newcommand{\ce}{\widetilde{e}}
\newcommand{\cphi}{\widetilde{\varphi}}
\newcommand{\cepsilon}{\widetilde{\varepsilon}}

\newcommand{\wt}{\mathrm{wt}}
\newcommand{\C}{\mathcal{C}}
\newcommand{\Q}{\mathcal{Q}}

\title{A local characterization of quasi-crystal graphs}
\author[A.J. Cain]{Alan J. Cain}
\address[A.J. Cain]{%
Center for Mathematics and Applications (NOVA Math)\\
NOVA School of Science and Technology\\
NOVA University of Lisbon\\
2829--516 Caparica\\
Portugal
}
\email{%
a.cain@fct.unl.pt
}
\thanks{This work is funded by national funds through the FCT - Fundação para a Ciência e a Tecnologia, I.P., under the
  scope of the projects UIDB/00297/2020 and UIDP/00297/2020 (Center for Mathematics and Applications).}

\author[A. Malheiro]{Ant\'onio Malheiro}
\address[A. Malheiro]{%
Center for Mathematics and Applications (NOVA Math) / Department of Mathematics \\
NOVA School of Science and Technology\\
NOVA University of Lisbon\\
2829--516 Caparica\\
Portugal
}
\email{%
ajm@fct.unl.pt
}

\author[F. Rodrigues]{F\'atima Rodrigues}
\address[F. Rodrigues]{%
Center for Mathematics and Applications (NOVA Math) / Department of Mathematics \\
NOVA School of Science and Technology\\
NOVA University of Lisbon\\
2829--516 Caparica\\
Portugal
}
\email{%
mfsr@fct.unl.pt
}

\author[I. Rodrigues]{Inês Rodrigues}
\address[I. Rodrigues]{%
Center for Mathematics and Applications (NOVA Math)\\
NOVA School of Science and Technology\\
NOVA University of Lisbon\\
2829--516 Caparica\\
Portugal
}
\email{%
ima.rodrigues@fct.unl.pt
}

\begin{document}
\begin{abstract}

It is provided a local characterization of quasi-crystal graphs, by presenting a set of local axioms, similar to the ones introduced by Stembridge for crystal graphs of simply-laced root systems. It is also shown that quasi-crystal graphs satisfying these axioms are closed under the tensor product recently introduced by Cain, Guilherme and Malheiro. It is deduced that each connected component of such a graph has a unique highest weight element, whose weight is a composition, and it is isomorphic to a quasi-crystal graph of semistandard quasi-ribbon tableaux.
\end{abstract}

\maketitle

\section{Introduction}

In the 1980's, Drinfeld \cite{Dri85} and Jimbo \cite{Jim85} independently introduced quantum groups,  which are quantized deformations of universal enveloping algebras of semisimple Lie algebras, and which  played a significant role in theoretical physics. 
Kashiwara, building upon their work, developed the theory of crystal bases and crystal graphs as a framework for studying representations of quantum groups \cite{Kas94,Kas90,Kas91}, at $q=0$ limit. Crystal bases are (informally) combinatorial objects associated with representations, and crystal graphs are directed graphs that encode the combinatorial data of crystal bases. Kashiwara has showed that, the crystal graph structure has very interesting properties such as being stable under tensor products \cite{Kas90}.

The theory of crystal bases has proven to be a valuable tool in representation theory and has found applications in various fields such as algebraic geometry, mathematical physics, and combinatorics. 
Indeed, crystal bases are related to Young tableaux through a categorification process that involves replacing the crystal operators of a crystal base by certain combinatorial operations, known as the Kashiwara operators. These operators correspond to edges in the crystal graph allowing for transitions between different tableaux. Through this connection, the combinatorial properties of Young tableaux, such as their shapes, content, and row and column insertion operations, can be related to the combinatorial structures of crystal bases \cite{KN94}. 

The  monoid whose elements are identified with Young tableaux is called the plactic monoid. It has origin in the works of Schensted \cite{Sch61}  and Knuth \cite{Knu70}, and it was later studied in depth by  Lascoux and Schützenberger \cite{LS81}. Kashiwara showed that the plactic monoid arises from the crystal bases associated with the vector representation of the quantized universal enveloping general linear Lie algebra.  It emerges as the quotient of a free monoid on a given alphabet $A_n=\{1,\ldots, n\}$ by a congruence that identifies words that have the same position in isomorphic components of the crystal graph \cite{KN94}.

The plactic monoid also plays a significant role in the theory of symmetric polynomials, particularly in connection with Schur polynomials. These polynomials, which serve as the irreducible polynomial characters of the general linear group $\mathrm{GL}_n(\mathbb{C})$, are indexed by shapes of Young tableaux with entries in the alphabet $A_n$. They form a basis for the ring of symmetric polynomials in $n$ indeterminates. The application of the plactic monoid has yielded the first rigorous proof of the Littlewood--Richardson rule, a combinatorial rule that expresses a product of two Schur polynomials as a linear combination of Schur polynomials \cite{Lit34}.

In addition to the classical plactic monoid, there exists another monoid known as the hypoplactic monoid that emerges in the realm of quasi-symmetric functions and non-commutative symmetric functions. It was first introduced by Krob and Thibon \cite{KT97} and studied in depth by Novelli \cite{Nov00}. The hypoplactic monoid provides an analogue to the classical plactic monoid, but with quasi-ribbon tableaux as elements. The  quasi-ribbon polynomials  serve as a basis for the ring of quasi-symmetric polynomials, analogous to how Schur polynomials form a basis for the ring of symmetric polynomials.

Cain and Malheiro \cite{CM17} introduced a purely combinatorial quasi-crystal structure for the hypoplactic monoid, similar to the crystal structure for the plactic monoid. 
They show that many of the intriguing connections observed between the crystal graph, Kashiwara operators, Young tableaux, and the plactic monoid are mirrored in the interaction of the analogous quasi-crystal graph, quasi-Kashiwara operators, quasi-ribbon tableaux, and the hypoplactic monoid. In particular, the hypoplactic monoid is defined as the quotient of the free monoid on the alphabet $A_n$ by the congruence that identifies words in the same position in isomorphic connected components of the quasi-crystal graph. 
Recently, Maas-Gariépy \cite{MG23} independently introduced an equivalent quasi-crystal structure by considering the decomposition of Schur functions into fundamental quasi-symmetric functions.
 
The first two authors, together with Guilherme \cite{CGM23},  introduced the concept of a hypoplactic congruence that can be defined for any seminormal quasi-crystal, leading to a broader notion of a hypoplactic monoid. They demonstrate that the hypoplactic monoid construction proposed by Cain and Malheiro can be viewed within the context of the hypoplactic monoid associated with the general linear Lie algebra. In the same paper, it is also  defined a quasi-tensor product of quasi-crystals, in a similar way as it was done for crystals. 

In this paper, the authors aim to further advance the theory of quasi-crystals in parallel with the existing theory of crystals.   This endeavor is motivated by the desire to extend the rich interplay between the crystal graph, Kashiwara operators, Young tableaux, and the plactic monoid to the realm of quasi-crystals. 

In \cite{Kas95}, Kashiwara  presented an abstract notion of crystal associated to a root system as an edge-coloured graph satisfying certain specific axioms, levering this notion  to a more general setting that goes beyond crystals of representations -- see \cite{BumpSchi17}. While explicit constructions of crystals exist for certain quantum algebras, such as those described in \cite{KN94} and \cite{Lit95}, the characterization of those arising from representations was obtained by Stembridge in \cite{Stem03}.
More specifically, Stembridge gave a set of local structural properties on  crystal graphs  of simply-laced root systems that permit to identify which  crystal graphs correspond to the crystal of a representation.
The simply-laced cases include all quantum Kac--Moody algebras having a Cartan matrix with off-diagonal entries of $0$ or $-1$. It is worth noting that these simply-laced crystals hold significant significance as they encompass all highest weight crystals of finite or affine type, which are of immense interest in the field. 

Other crystal-like structures also exhibit local characterizations. For instance, Gillespie and Levinson \cite{GL19} provided local axioms for a crystal of shifted tableaux, Gillespie, Hawkes, Poh and Schilling \cite{GHPS20} gave a characterization of crystals for the quantum queer superalgebra, building on local axioms introduced by Assaf and O\u{g}uz \cite{AO20}, and Tsuchioka \cite{Tsu21} introduced a local characterization of $B_2$ regular crystals. 

Paralleling the previous work, the authors present in this paper a set of local axioms, similar to those presented by Stembridge for crystals, that characterize quasi-crystal graphs of simply-laced root systems that arise from the quasi-crystal of type $A_n$. This characterization answers a question (Question 1)  posed by the referee of \cite{CM17}.

This paper is organized as follows. In Section \ref{sec:background} we recall the notion of crystals, focusing on Stembridge crystals, and quasi-crystals.
In Section \ref{sec:local_axs}, we introduce local axioms for quasi-crystal graphs (Definition \ref{def:local_axioms}), and prove that connected quasi-crystals satisfying these axioms are completely characterized by their unique highest weight elements (Theorems \ref{thm:uniq_hw} and \ref{thm:wt_iso}). We also prove that the quasi-crystal graphs satisfying the said axioms are closed under the quasi-tensor product introduced in \cite{CGM23}.
In Section \ref{sec:stem_crystals_quasi}, we introduce an algorithm to obtain quasi-crystal graphs satisfying the local axioms from a connected Stembridge crystal (Theorem \ref{thm:crystal_quasi_ax}).

\section{Background}\label{sec:background}
We begin by recalling the notion of (abstract) crystals and Stembridge crystals, following mainly \cite{BumpSchi17}. Then, we recall the notion of quasi-crystals, first introduced in \cite{CM17} and further developed in \cite{CGM23}.

\subsection{Crystals}

\begin{defin}\label{def:crystal}
Let $\Phi$ be a root system, with weight lattice $\Lambda$, index set $I$ and simple roots $\alpha_i$, for $i \in I$.
A \emph{crystal} of type $\Phi$  is a non-empty set $\mathcal{C}$ together with maps 
$\ce_i, \cf_i: \C \longrightarrow \mathcal{C} \sqcup \{ \bot \}$, $\cepsilon_i, \cphi_i : \C \longrightarrow \mathbb{Z} \sqcup \{-\infty\}$, and $\wt: \C \longrightarrow \Lambda$ 
for $i \in I := \{1, \ldots, n-1\}$, satisfying the following:
\begin{description}
\item[C1.] For any $x,y \in \mathcal{C}$, $\ce_i (x) = y$ if and only $x = \cf_i (y)$, and in that case,
\begin{align*}
\wt(y) &= \wt(x) + \alpha_i,\\ 
\cepsilon_i(y) &= \cepsilon_i (x) - 1,\\ 
\cphi_i (y) &= \cphi_i (x)+1,
\end{align*}

\item[C2.] $\cphi_i (x) = \cepsilon_i (x) + \langle \wt(x), \alpha_i^{\vee} \rangle$. 

\item[C3.] If $\cepsilon_i (x) = -\infty$, then $\ce_i (x) = \cf_i (x) = \bot$. 
\end{description}
\end{defin}

The maps $\ce_i$ and $\cf_i$ are called the  \emph{Kashiwara operators} (respectively, the \emph{raising} and \emph{lowering} operators), $\cepsilon_i$ and $\cphi_i$ are the \emph{length functions} and $\wt$ is the \emph{weight function}.
A crystal is said to be \emph{seminormal} if
\begin{align*}
\cepsilon_i (x) &= \max \{k: \ce_i^{\,k} (x) \neq \bot\},\\
\cphi_i (x) &= \max \{k: \cf_i^{\,k} (x) \neq \bot\},
\end{align*}
for all $i \in I$, $x \in \C$.
In particular, in a seminormal crystal, we have $\cepsilon_i(x), \cphi_i (x) \geq 0$, for all $x \in \C$, $i \in I$. We say that $x \in \C$ is a \emph{highest weight element} if $\ce_i(x) = \bot$, for all $i \in I$. If $\C$ is seminormal, this is equivalent to have $\cepsilon_i (x)= 0$, for all $i \in I$.

We associate a crystal to its \emph{crystal graph}, a directed graph, whose edges are labelled in $I$ and whose vertices are weighted in $\Lambda$, and such that there exists an $i$-labelled edge $y \overset{i}{\longrightarrow} x$ if and only if $\cf_i (y) = x$.

\begin{defin}
Let $\Phi$ be a simply-laced root system. A crystal of type $\Phi$ is a \emph{weak Stembridge crystal} if the following axioms are satisfied, for all $i, j \in I$ such that $i \neq j$:
\begin{enumerate}
\item[\textbf{S1.}] If $\ce_i (x)=y$, then $\cepsilon_j (y)$ is equal to either $\cepsilon_j(x)$ or $\cepsilon_j (x)+1$, and the latter happens only if $\alpha_i$ and $\alpha_j$ are orthogonal roots.

\item[\textbf{S2.}] If $\ce_i (x) = y$ and $\cepsilon_j (y) = \cepsilon_j (x) > 0$, then 
$$\ce_i \ce_j (x) = \ce_j \ce_i (x) \neq \bot,$$
and $\cphi_i (\ce_j (x)) = \cphi_i (x)$.

\item[\textbf{S2$'$.}] If $\cf_i (x) = y$ and $\cphi_j (y) = \cphi_j (x) > 0$, then 
$$\cf_i \cf_j (x) = \cf_j \cf_i (x) \neq \bot,$$
and $\cepsilon_i (\cf_j (x)) = \cepsilon_i (x)$.

\item[\textbf{S3.}] If $\ce_i (x) = y$ and $\ce_j (x)=z$, and
$\cepsilon_j (y) = \cepsilon_j (x)+1$ and $\cepsilon_i (z)= \cepsilon_i (x)+1$, then $$\ce_i \ce_j^{\,2} \ce_i (x) = \ce_j \ce_i^{\,2} \ce_j (x) \neq \bot,$$
and $\cphi_i (\ce_j (x)) = \cphi_i (\ce_j^{\,2} \ce_i (x))$ and $\cphi_j (\ce_i (x)) = \cphi_j (\ce_i^{\,2} \ce_j (x))$.

\item[\textbf{S3$'$.}] If $\cf_i (x) = y$ and $\cf_j (x)=z$, and
$\cphi_j (y) = \cphi_j (x)+1$ and $\cphi_i (z)= \cphi_i (x)+1$, then $$\cf_i \cf_j^{\,2} \cf_i (x) = \cf_j \cf_i^{\,2} \cf_j (x) \neq \bot,$$
and $\cepsilon_i (\cf_j (x)) = \cepsilon_i (\cf_j^{\,2} \cf_i (x))$ and $\cepsilon_j (\cf_i (x)) = \cepsilon_j (\cf_i^{\,2} \cf_j (x))$.
\end{enumerate}
A \emph{Stembridge crystal} is a weak Stembridge crystal that is also seminormal. 
\end{defin}

\begin{prop}[{\cite[Proposition 4.5]{BumpSchi17}}]\label{prop:stem_prop}
Let $\C$ be a crystal graph satisfying axiom \textbf{S1} and let $\ce_i (x) = y$. Then, exactly one of the following possibilities is true:
\begin{enumerate}
\item $\cepsilon_j (y) = \cepsilon_j (x)$, $\cphi_j (y) = \cphi_j (x)-1$, for $\langle \alpha_i, \alpha_j \rangle = -1$.
\item $\cepsilon_j (y) = \cepsilon_j (x)+1$, $\cphi_j (y) = \cphi_j (x)$, for $\langle \alpha_i, \alpha_j \rangle = -1$.
\item $\cepsilon_j (y) = \cepsilon_j (x)$, $\cphi_j (y) = \cphi_j (x)$, for $\langle \alpha_i, \alpha_j \rangle = 0$.
\end{enumerate}
\end{prop}

\subsection{Quasi-crystals}

Consider $\mathbb{Z}\sqcup\{-\infty, +\infty\}$ the set of integers with two additional symbols, a minimal element $-\infty$ and a maximal element $+\infty$. In this set consider also the usual addition between integers, and set $m+(-\infty)=(-\infty)+m= - \infty$ and $m+(+\infty)=(+\infty)+m= + \infty$, for all $m \in \mathbb{Z}$.

\begin{defin}[{\cite[Definition 3.1]{CGM23}}]\label{def:quasicrystal}
A \emph{quasi-crystal} of type $\Phi$  is a non-empty set $\Q$ together with maps 
$\qe_i, \qf_i: \Q \longrightarrow \Q \sqcup \{ \bot \}$, $\qepsilon_i, \qphi_i : \Q \longrightarrow \mathbb{Z} \sqcup \{-\infty, + \infty\}$, and $\wt: \Q \longrightarrow \Lambda$ 
for $i \in I := \{1, \ldots, n-1\}$, satisfying the following:
\begin{description}
\item[Q1.] For any $x,y \in \Q$, $\qe_i (x) = y$ if and only $x = \qf_i (y)$, and in that case,
\begin{align*}
\wt(y) &= \wt(x) + \alpha_i,\\ 
\qepsilon_i(y) &= \qepsilon_i (x) -1,\\ 
\qphi_i (y) &= \qphi_i (x)+1.
\end{align*}

\item[Q2.] $\qphi_i (x) = \qepsilon_i (x) + \langle \wt(x), \alpha_i^{\vee} \rangle$.
\item[Q3.] If $\qepsilon_i (x) = - \infty$, then $\qe_i (x) = \qf_i (x) = \bot$.
\item[Q4.] If $\qepsilon_i (x) = + \infty$, then $\qe_i (x) = \qf_i (x) = \bot$.
\end{description}
\end{defin}
We use the same terminology as of crystals, except the maps $\qe_i$ and $\qf_i$ are called the \emph{quasi-Kashiwara operators}.
A quasi-crystal is said to be \emph{seminormal} if
\begin{align*}
\qepsilon_i (x) &= \max \{k: \qe_i^k (x) \neq \bot\},\\
\qphi_i (x) &= \max \{k: \qf_i^k (x) \neq \bot\},
\end{align*}
for all $i \in I$, $x \in \C$, whenever $\qepsilon_i (x) \neq + \infty$.  In particular, in a seminormal quasi-crystal $\qepsilon_i(x),\qphi_i(x)\in\mathbb{Z}_{\geq 0} \sqcup \{+ \infty\}$.

It follows from Definition \ref{def:quasicrystal} that a crystal is a quasi-crystal such that $\qepsilon_i (x), \qphi_i (x) \neq + \infty$, for all $x \in \C, i \in I$ \cite[Remark 3.2]{CGM23}. We say that $x \in \Q$ is a \emph{highest weight element} if $\qe_i (x) = \bot$ for all $i \in I$. Note that, in a seminormal quasi-crystal, this is not equivalent to have $\qepsilon_i (x) = 0$ for all $i \in I$, as one might have $\qe_j (x) = \bot$ and $\qepsilon_j (x) = + \infty$.

Similarly to crystals, we associate a quasi-crystal with its \emph{quasi-crystal graph}, defined in the same way as the crystal graph, but such that $x \in \Q$ has an $i$-labelled loop if and only if $\qepsilon_i (x) = \qphi_i (x)=+ \infty$.

\section{Local axioms for quasi-crystals}\label{sec:local_axs}

In what follows, we will consider quasi-crystals of type $A_{n-1}$, where the weight lattice is $\Lambda = \mathbb{Z}^n$, the index set $I = \{1 < \cdots < n-1\}$, and the simple roots are $\alpha_i = (0, \ldots, 0, 1, -1, 0, \ldots, 0)$, and satisfy $\alpha_i^{\vee} = \alpha_i$, for $i\in I$. We will also use the symbols $\Q$, $\qe_i$, $\qf_i$, $\qepsilon_i$, $\qphi_i$ and $\wt$ as in Definition~\ref{def:quasicrystal}. 

\begin{obs}
We focus our study on type $A_{n-1}$ quasi-crystals, because, unlike the crystal case, the other simply-laced types do not exhibit the same nice properties. For instance, connected quasi-crystals of type $D_n$ might have more that one highest weight element \cite{CGM23}.
\end{obs}

We remark that, for seminormal crystals (or, more generally, for upper seminormal ones, where it is only required that $\qepsilon_i (x) = \max \{k: \qe_i^k (x) \neq \bot\}$), the condition $\cepsilon_i (x) > 0$ is equivalent to $\ce_i (x) \neq \bot$. For seminormal quasi-crystals, one may have $\qepsilon_i (x) = + \infty > 0$ and, by \textbf{Q4}, $\qe_i (x) = \bot$. Therefore, in the following statements, differing from \cite{BumpSchi17}, we will require specifically that $\qe_i (x) \neq \bot$, instead of $\qepsilon_i (x) > 0$. For this reason, for certain results we do not require that the quasi-crystal graphs are seminormal, but rather that $\qepsilon_i (x), \qphi_i (x) \in \mathbb{Z}_{\geq 0} \sqcup \{+ \infty\}$ (which is satisfied by seminormal quasi-crystals). 

\begin{defin}[Local quasi-crystal axioms]\label{def:local_axioms}
Let $\Q$ be a quasi-crystal graph. Let $i,j\in I$ and $x,y\in \Q$.
\begin{description}
\item[LQ1.] For $i+1\in I$, $\qepsilon_i(x) = 0 \Leftrightarrow \qphi_{i+1}(x)=0$. 

\item[LQ2.] If $\qe_i(x)=y$, then:
\begin{enumerate}
\item $\qepsilon_j(x) = \qepsilon_j(y)$, for $|i-j|>1$.
\item For $i+1\in I$, $\qepsilon_{i+1}(x) \neq \qepsilon_{i+1}(y) \Leftrightarrow \big( \qepsilon_{i+1}(x) = + \infty \wedge  \qepsilon_i(y) =0 \big) \Rightarrow \qepsilon_{i+1}(y)\neq 0$.
\item For $i-1\in I$, $\qphi_{i-1}(x) \neq \qphi_{i-1}(y) \Leftrightarrow \big( \qphi_{i-1}(y) = + \infty \wedge  \qphi_i(x) = 0 \big) \Rightarrow \qphi_{i-1}(x)\neq 0$.
\end{enumerate}

\item[LQ3.] For $i\neq j$,  if both $\qe_i(x)$ and $\qe_j(x)$ are defined, then $\qe_i \qe_j(x)=\qe_j \qe_i(x)\neq \bot$.

\item[LQ3$'$.] For $i\neq j$, if both $\qf_i(x)$ and $\qf_j(x)$ are defined, then $\qf_i \qf_j(x)=\qf_j \qf_i(x)\neq \bot$.
\end{description}
\end{defin}

The three cases of axiom \textbf{LQ2} are depicted in Figures \ref{fig:LQ2_1}, \ref{fig:LQ2_2} and \ref{fig:LQ2_3}.

\begin{figure}[h]
\begin{center}
\includegraphics[scale=1]{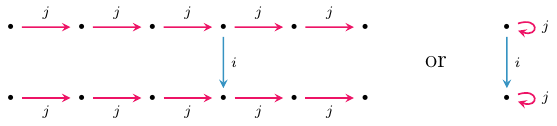}
\caption{Illustration of axiom \textbf{LQ2} (1). The blue labels denote $\qf_{i}$ and the red ones $\qf_j$ or $j$-labelled loops.}
\label{fig:LQ2_1}
\end{center}
\end{figure}

\begin{figure}[h]
\begin{center}
\includegraphics[scale=1]{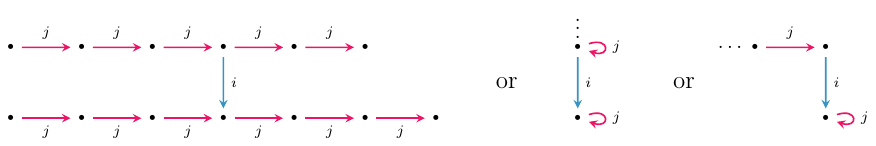}
\caption{Illustration of axiom \textbf{LQ2} (2). The blue labels denote $\qf_{i}$ and the red ones $\qf_{i+1}$ or $(i+1)$-labelled loops.}
\label{fig:LQ2_2}
\end{center}
\end{figure}

\begin{figure}[h]
\begin{center}
\includegraphics[scale=1]{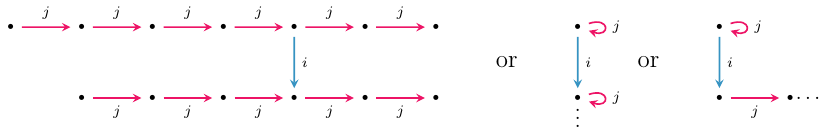}
\caption{Illustration of axiom \textbf{LQ2} (3). The blue labels denote $\qf_{i}$ and the red ones $\qf_{i-1}$ or $(i-1)$-labelled loops.}
\label{fig:LQ2_3}
\end{center}
\end{figure}

We remark that, in the previous definition, the condition in \textbf{LQ2} (1) can be replaced with $\qphi_j$. We prove this equivalence in the next result.

\begin{lem}\label{lem:lem_ij}
Let $\Q$ be a quasi-crystal graph satisfying \emph{\textbf{LQ1}} and \emph{\textbf{LQ2}}. Let $i,j\in I$ and $x,y\in \Q$, and suppose that $\qe_i(x) = y$.
\begin{enumerate}
\item If $|i-j|>1$, then $\big( \qepsilon_j(y)=\qepsilon_j(x) \Leftrightarrow \qphi_j(y)=\qphi_j(x)\big)$.
\item If $\qepsilon_{i+1} (x) \neq + \infty$, then $\big(\qepsilon_{i+1} (y) = \qepsilon_{i+1} (x) \Leftrightarrow \qphi_{i+1} (y) = \qphi_{i+1} (x)-1\big)$.
\item If $\qepsilon_{i-1} (x) \neq + \infty$, then $\big(\qepsilon_{i-1} (y) = \qepsilon_{i-1} (x)+1 \Leftrightarrow \qphi_{i-1} (y) = \qphi_{i-1} (x)\big)$.
\end{enumerate}
\end{lem}

\begin{proof}
Suppose that $|i-j|>1$. Then, since $\Q$ is of type $A_{n-1}$, we have $\langle \alpha_i, \alpha_j \rangle = 0$. Suppose that $\qepsilon_j (y) = \qepsilon_j (x)$. Then,
\begin{align*}
\qphi_j(y) &= \qepsilon_j(y) + \langle \wt(y), \alpha_j \rangle & (\textrm{by \textbf{Q2} and since $\alpha_i^{\vee}=\alpha_i$})\\
&= \qepsilon_j (x) + \langle \wt(x) + \alpha_i, \alpha_j \rangle & (\textrm{by \textbf{Q1}})\\
&= \qepsilon_j(x) + \langle \wt(x), \alpha_j \rangle + \langle \alpha_i, \alpha_j\rangle \\
&= \qepsilon_j(x) + \langle \wt(x), \alpha_j \rangle \\
&=  \qphi_j(x) & (\textrm{by \textbf{Q2} and since  $\alpha_i^{\vee}=\alpha_i$}).
\end{align*}
Similarly, $\qphi_j (y) = \qphi_j (x)$, implies that $\qepsilon_j (y) = \qepsilon_j (x)$.

Now suppose that $\qepsilon_{i+1} (x) \neq + \infty$, and suppose that $\qepsilon_{i+1} (y) = \qepsilon_{i+1} (x)$. Then, 
\begin{align*}
\qphi_{i+1} (y) &= \qepsilon_{i+1}(y) + \langle \wt(y), \alpha_{i+1} \rangle & (\textrm{by \textbf{Q2} and since  $\alpha_i^{\vee}=\alpha_i$})\\
&= \qepsilon_{i+1} (x) + \langle \wt(x) + \alpha_i, \alpha_{i+1} \rangle & (\textrm{by \textbf{Q1}})\\
&= \qepsilon_{i+1} (x) + \langle \wt(x), \alpha_{i+1} \rangle + \langle \alpha_i, \alpha_{i+1} \rangle\\
&= \qepsilon_{i+1} (x) + \langle \wt(x), \alpha_{i+1} \rangle - 1\\
&= \qphi_{i+1} (x)-1 & (\textrm{by \textbf{Q2} and since  $\alpha_i^{\vee}=\alpha_i$})
\end{align*}
and similarly, one shows that $\qphi_{i+1} (y) = \qphi_{i+1} (x)-1$ implies that $\qepsilon_{i+1} (y) = \qepsilon_{i+1} (x)$. The proof for the third case is analogous.
\end{proof}

\begin{prop}\label{prop:local_ax}
Let $\Q$ be a quasi-crystal graph satisfying \emph{\textbf{LQ1}} and \emph{\textbf{LQ2}} and such that $\qepsilon_i (x)$, $\qphi_i (x) \in \mathbb{Z}_{\geq 0} \sqcup \{+ \infty\}$, for all $x \in Q, i \in I$. Let $i,j\in I$ and $x,y\in \Q$, and suppose that $\qe_i(x) = y$. Then:

\begin{enumerate}
\item If $|i-j|>1$, $\qepsilon_j(y)=\qepsilon_j(x)$ and $\qphi_j(y)=\qphi_j(x)$.

\item If $i-j=-1$, then:
\begin{enumerate}
\item If $\qepsilon_{i+1}(x) \neq + \infty$, then $\qepsilon_{i+1}(y)=\qepsilon_{i+1}(x)$ and $\qphi_{i+1}(y)=\qphi_{i+1}(x)-1$.
\item If $\qepsilon_{i+1}(x)=+\infty$ and $\qepsilon_i(y)>0$, then $\qepsilon_{i+1}(y)=\qepsilon_{i+1}(x)=+\infty$ and $\qphi_{i+1}(y)=\qphi_{i+1}(x)=+\infty$.
\item If $\qepsilon_{i+1}(x)=+\infty$ and $\qepsilon_i(y)=0$, then $\qepsilon_{i+1}(y)= {- \langle \wt(y), \alpha_{i+1} \rangle} > 0$ and $\qphi_{i+1}(y)=0$.
\end{enumerate}

\item If $i-j = 1$, then:
\begin{enumerate}
\item If $\qphi_{i-1}(y) \neq + \infty$, then $\qepsilon_{i-1}(x)=\qepsilon_{i-1}(y)-1$ and $\qphi_{i-1}(x)=\qphi_{i-1} (y)$.
\item If $\qphi_{i-1}(y)=+\infty$ and $\qphi_i(x)>0$, then $\qepsilon_{i-1}(x)=\qepsilon_{i-1}(y)=+\infty$ and $\qphi_{i-1}(x)=\qphi_{i-1}(y)=+\infty$.
\item If $\qphi_{i-1}(y)=+\infty$ and $\qphi_i(x)=0$, then $\qepsilon_{i-1}(x)=0 $ and $\qphi_{i-1}(x)= {\langle \wt(x), \alpha_{i-1} \rangle} >0$.
\end{enumerate}
\end{enumerate}
\end{prop}

We remark that, in the previous result, exactly one of the cases holds.

\begin{proof}[Proof of Proposition \ref{prop:local_ax}]
Since $\qe_i(x) = y$, we have $\wt(y) = \wt(x) + \alpha_i$, by \textbf{Q1}. 

The first condition is a direct consequence of Lemma \ref{lem:lem_ij}.

Now suppose that $j = i+1$ (the proof for $j=i-1$ is analogous). Then, $\langle \alpha_i, \alpha_{i+1} \rangle = -1$. If $\qepsilon_{i+1} (x) \neq + \infty$, the result follows  from Lemma \ref{lem:lem_ij}. If $\qepsilon_{i+1} (x) = + \infty$, and since $\qepsilon_{i}(y) \geq 0$, we have two cases to consider:

\begin{enumerate}
\item If $\qepsilon_i(y)>0$, axiom \textbf{LQ2} implies that $\qepsilon_{i+1} (y) = \qepsilon_{i+1} (x) = + \infty$. Consequently, by \textbf{Q2}, we have $\qphi_{i+1} (y) = +\infty$.

\item If $\qepsilon_i(y)=0$, then \textbf{LQ2} implies that $\qepsilon_{i+1} (x) \neq \qepsilon_{i+1}(y)$ and $\qepsilon_{i+1} (y) > 0$. By axiom \textbf{LQ1}, $\qepsilon_i(y) = 0$ implies that $\qphi_{i+1} (y) = 0$. In particular, by \textbf{Q2}, we have that $\qepsilon_{i+1}(y)$ and 
 $\qphi_{i+1} (y)$ are finite, and thus, since $\alpha_i^{\vee}=\alpha_i$,
\[
\qepsilon_{i+1} (y) = \qphi_{i+1} (y) - \langle \wt(y), \alpha_{i+1} \rangle = - \langle \wt(y), \alpha_{i+1} \rangle.
\qedhere \]
\end{enumerate}
\end{proof}

The next result is illustrated in Figure \ref{fig:local_infs}.

\begin{cor}\label{cor:infs}
Let $\Q$ be a seminormal quasi-crystal graph satisfying \emph{\textbf{LQ1}} and \emph{\textbf{LQ2}}, and suppose that $\qe_i (x) = y$, for some $x,y\in \Q$ and $i\in I$.
\begin{enumerate}
\item Let $i+1\in I$. If $\qepsilon_{i+1}(y) = + \infty$, then $\qepsilon_{i+1} (x) = + \infty$ and, there exists $k>0$ such that 
$\qepsilon_{i+1} (\qe_i^k (y)) \notin \{0, + \infty\}$, and for $0\leq l<k$, we have $\qepsilon_i(\qe_i^l(y))\notin \{0, + \infty\}$.

\item  Let $i-1\in I$. If $\qphi_{i-1} (x) = + \infty$, then $\qphi_{i-1} (y) = + \infty$ and there exists $k>0$ such that  $\qphi_{i-1}  (\qf_i^k (x)) \notin \{0,+\infty\}$, and  for $0\leq l<k$, we have $\qphi_{i-1}(\qf_i^l(y))\notin \{0, + \infty\}$.
\end{enumerate}
\end{cor}

\begin{proof}
We prove the first statement; the second one is proved similarly. Suppose that $\qepsilon_{i+1} (y) = + \infty$. If $\qepsilon_{i+1} (x) \neq + \infty$, then axiom \textbf{LQ2} (2) implies that $\qepsilon_{i+1} (y)= \qepsilon_{i+1} (x) \neq + \infty$, which is a contradiction. Therefore, $\qepsilon_{i+1} (x)= +\infty$.

Now suppose that $\qepsilon_i (y)=0$. By \textbf{LQ1}, we have $\qphi_{i+1} (y)=0$, which implies that $\qepsilon_{i+1} (y)$ is finite, contradicting the hypothesis. Thus, $\qepsilon_i (y) \neq 0$. By \textbf{Q1}, $\qe_i (x) = y$ implies that $\qf_i (y)=x$, hence $\qf_i (y) \neq \bot$, and by \textbf{Q4}, $\qepsilon_i (y) \neq + \infty$. Therefore, $\qepsilon_i (x) \not\in \{0, + \infty\}$,
and thus, there exists $k > 0$ such that $\qepsilon_i (y)=k$. Since $\Q$ is seminormal, we have $\qe_i^l (y) \neq \bot$, for all $l \leq k$ and $\qe_i^s (y)= \bot$, for $s > k$. Therefore, $\qepsilon_i (\qe_i^l (y)) \not\in \{0, +\infty\}$, for $l < k$. 
Now, iteratively, applying \textbf{LQ2} (2), we obtain  $\qepsilon_{i+1} (\qe_i^l (y))=\qepsilon_{i+1}  (y) = + \infty$, for $1\leq l < k$.
Finally, since $\qepsilon_i(\qe_i^k(y))=0$ and  $\qepsilon_{i+1} (\qe_i^{k-1} (y))=+\infty$ we conclude by \textbf{LQ2} (2), that  $\qepsilon_{i+1} (\qe_i^{k}(y)) \notin \{0,+ \infty\}$.
\end{proof}

\begin{figure}[h]
\begin{center}
\includegraphics[scale=1]{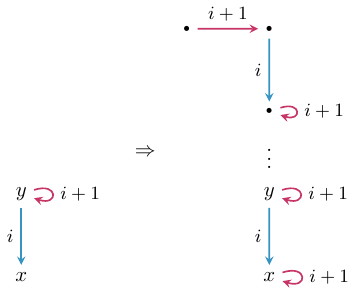}
\hspace{1.5cm}
\includegraphics[scale=1]{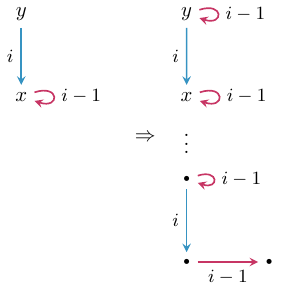}
\end{center}
\caption{Illustration of Corollary \ref{cor:infs}.}
\label{fig:local_infs}
\end{figure}

We have recalled that a crystal $\C$ is a quasi-crystal such that $\qepsilon_i (x) \neq + \infty$ for all $x \in \C, i \in I$ \cite{CGM23}. Therefore, we have the following result.

\begin{prop}
Let $\Q$ be a crystal graph satisfying the local quasi-crystal axioms of Definition \ref{def:local_axioms}. Then $\Q$ is a weak Stembridge crystal.
\end{prop}

\begin{proof}
Suppose that $\qe_i (x) = y$, for $x, y \in \Q$, $i \in I$ and let $j \in I$ be such that $j \neq i$. Since $\Q$ is a crystal, we have $\qepsilon_j (x) \neq + \infty$. Thus, if $|i-j|>1$ or $j=i+1$, by axiom \textbf{LQ2} we have $\qepsilon_j (y) = \qepsilon_j (x)$. Otherwise, if $j=i-1$ (and, in this case, $\alpha_j$ and $\alpha_i$ are orthogonal roots), we have $\qepsilon_j (y) = \qepsilon_j (x)+1$. Therefore, $\Q$ satisfies axiom \textbf{S1}.

Now suppose that $\qe_i (x)=y$ and $\qe_j (x)=z$ are both defined. Then, axiom \textbf{LQ3} ensures that $\qe_i \qe_j (x) = \qe_j \qe_i (x) \neq \bot$. We now consider three cases:
\begin{enumerate}
\item If $|i-j|>1$, we have $\qepsilon_i (z) = \qepsilon_i (x)>0$ and $\qepsilon_j(y)=\qepsilon_j (x)>0$, and thus, by Lemma \ref{lem:lem_ij}, we have $\qphi_i (z) = \qphi_i (x)$ and $\qphi_j (y) = \qphi_j (x)$. 
\item If $j=i+1$, and since $\qe_i (x)=y$ and $\qepsilon_j (x) \neq + \infty$ (because $\Q$ is a crystal and thus $\qepsilon_i (x) \neq + \infty$, for any $x \in Q$, $i \in I$), it follows from axiom \textbf{LQ2} (2) that $\qepsilon_j (y) = \qepsilon_j (x) > 0$. By the same reasoning, since $\qe_j (x)=z$ and $\qphi_i (z) \neq + \infty$, axiom \textbf{LQ2} (3) ensures that $\qepsilon_i (z) = \qepsilon_i (x)+1$, and thus, by Lemma \ref{lem:lem_ij}, we have $\qphi_i (z)=\qphi_i (x)$.
\item If $j=i-1$, since $\qe_i (x)=y$ and $\qphi_j (y) \neq + \infty$, axiom \textbf{LQ2} (3) implies that $\qphi_j (y)=\qphi_j (x)$, and thus, by Lemma \ref{lem:lem_ij}, we have $\qepsilon_j (y)=\qepsilon_j (x)+1$. Similarly, since $\qe_j (x)=z$ and $\qepsilon_i (x) \neq + \infty$, axiom \textbf{LQ2} (2) implies that $\qepsilon_i (z) = \qepsilon_i (x) > 0$.
\end{enumerate}
In all cases, $\Q$ satisfies axiom \textbf{S2}. Moreover, the conditions $\qepsilon_i (z) = \qepsilon_i (x)+1$ and $\qepsilon_j (y) = \qepsilon_j (x)+1$ do not occur simultaneously, hence the hypotheses of axiom \textbf{S3} never occur. With similar reasoning, we obtain the same conclusions for the dual axioms.
\qedhere
\end{proof}

Recall that a highest weight element in a quasi-crystal is an element $x$ such that $\qe_i(x) = \bot$, for all $i \in I$.
We define a partial order in $\Q$ where $x \prec y$ if there exists a sequence $i_1, \dots, i_N$ in $I$ such that $\qe_{i_1} \dots \qe_{i_N} (x)=y$.
Similarly to the crystal case, we say that a quasi-crystal graph $\Q$ is \emph{bounded above} if, for every $x \in \Q$, there exists and highest weight element $x_h \in \Q$ such that $x \preceq x_h$.

\begin{thm}\label{thm:uniq_hw}
Let $\Q$ be a non-empty and bounded above connected quasi-crystal graph satisfying the local quasi-crystal axioms. Then, $\Q$ has a unique highest weight element. 
\end{thm}

\begin{proof}
Let $S$ be the subset of vertices $w \in \Q$ for which there exist distinct highest weight elements $x$ and $y$, such that $ w \preceq x$ and $w \preceq y$. Suppose that $S \neq \emptyset$ and let $w_0$ be a maximal element of $S$. Then, by maximality, there are vertices $x, y \in \Q$, such that $w_0 \prec x$ and $w_0 \prec y$, where $x$ and $y$ are each comparable to only one highest weight element, say $x_h$ and $y_h$, respectively. Without loss of generality, assume that $x = \qe_i (w_0)$ and $y = \qe_j (w_0)$, for $i, j \in I$. Then, by axiom \textbf{LQ3}, there exists $z \in C$ such that  $\qe_i \qe_j (w_0) = \qe_j \qe_i (w_0) = z$. Then, we have $x \prec z$, and by hypothesis, $x$ is comparable to a unique highest weight element $x_h$. Thus, we have $z \preceq x_h$. By the same reasoning, we have $z \preceq y_h$, and thus $z \in S$. 
But since $w_0 \prec x$, we have $w_0 \prec x \prec z$, which contradicts the maximality of $w_0$. Therefore, $S = \emptyset$.
\end{proof}

\begin{thm}\label{thm:wt_iso}
Let $\Q$ and $\Q'$ be connected components of seminormal quasi-crystal graphs satisfying the local quasi-crystal axioms, with highest weight elements $u$ and $u'$, respectively. Then, if $\wt(u)=\wt(u')$, there exists an isomorphism between $\Q$ and $\Q'$. 
\end{thm}

To prove Theorem \ref{thm:wt_iso}, we recall the notion of \emph{rank} of $x \in \Q$ \cite[Section 4.4]{BumpSchi17}, where $\Q$ is a connected component of a quasi-crystal graph with unique highest weight element $u$, which is defined as $\mathrm{rank} (x) := \langle \wt(u) - \wt(x), \rho \rangle$, where $\rho$ is any vector such that $\langle \alpha_i, \rho \rangle = 1$, for all $i \in I$. This is well defined, and, in particular, if $\qe_{i_N} \cdots \qe_{i_1} (x) = u$, and all $\qe_{i_1} (x), \qe_{i_2} \qe_{i_1} (x), \ldots, \qe_{i_N} \cdots \qe_{i_1} (x)$ are defined, then $\mathrm{rank}(x) = N$.

\begin{proof}[Proof of Theorem \ref{thm:wt_iso}]
Let $\Omega$ be the set of subsets $S \subseteq \Q$ such that:
\begin{itemize}
\item $u \in S$.
\item If $x \in S$ and $\qe_i (x) \in Q$, then $\qe_i (x) \in S$.
\item There exists a subset $S' \subseteq \Q'$ and a bijection $\theta : S \longrightarrow S'$ such that $\theta(u) = u'$ and, given $x \in S$, then, for every $i \in I$,
$$\big(\qe_i (x) \neq \bot \Leftrightarrow \qe_i (\theta(x)) \neq \bot \big) \wedge  \big( \qe_i(x) \neq \bot \Rightarrow \theta (\qe_i (x)) = \qe_i (\theta (x)) \big).$$
\end{itemize}
We have $\Omega \neq \emptyset$, since $\{u\} \in \Omega$. Let $S$ be a maximal element of $\Omega$, with respect to set inclusion. We will show that $S = \Q$. We claim that $\theta$ preserves the length and weight functions, that is, for all $x \in S$ and $i \in I$, $\qepsilon_i (\theta(x)) = \qepsilon_i (x)$, $\qphi_i (\theta(x)) = \qphi_i (x)$ and $\wt (\theta(x)) = \wt (x)$. Given $x \in S$ and $i \in I$, there exist $i_1, \ldots, i_N \in I$, for some $N$, such that 
$$\qe_{i_N} \cdots \qe_{i_1} (x) = u.$$
By the definition of $S$, we have 
$$u' = \theta (u) = \theta (\qe_{i_N} \cdots \qe_{i_1} (x))
= \qe_{i_N} \cdots \qe_{i_1} ( \theta (x) )
$$
and thus,
$$
\wt (\theta (x)) = \wt (u') - \sum\limits_{j=1}^N \alpha_{i_j}
= \wt (u) - \sum\limits_{j=1}^N \alpha_{i_j}
= \wt (x).$$
Given $x \in S$, since $\Q$ is seminormal, if $\qepsilon_i (x) \neq + \infty$, then $\qepsilon_i (x) = \max \{k: \qe_i^k(x) \neq \bot\}$. By the definition of $S$, we have $\qe_i^k (x) \in S$, for $k = 1, \ldots, \qepsilon_i (x)$. In particular, $\qe_i^k (\theta(x)) = \theta (\qe_i^k (x))$. Thus, we have
\begin{align*}
\qepsilon_i (\theta(x)) &= \max \{k: \qe_i^k (\theta(x)) \neq \bot\}\\
&= \max \{k: \theta(\qe_i^k (x)) \neq \bot\}\\
&= \max \{k: \qe_i^k (x) \neq \bot\} = \qepsilon_i (x).
\end{align*}
Consequently, $\qepsilon_i (x) = + \infty$ if and only if $\qepsilon_i (\theta(x))=+ \infty$, and therefore, $\qepsilon_i (x) = \qepsilon_i (\theta(x))$. From \textbf{Q2}, we get $\qphi_i (x) = \qphi_i (\theta(x))$.

Now suppose that $\Q \neq S$, and let $z \in \Q \setminus S$ be an element of minimal rank. Since $z \not\in S$, then $z \neq u$, and thus, there exists $i \in I$ such that $\qe_i (z) \neq \bot$. Moreover, $\mathrm{rank} (\qe_i (z)) < \mathrm{rank} (z)$, hence $\qe_i (z) \in S$, as $z$ has minimal rank. Therefore, $\qphi_i ( \theta(\qe_i(z))) = \qphi_i (\qe_i (z)) > 0$, and thus, there exists $z' \in \Q'$ such that $\qe_i (z') = \theta (\qe_i (z))$, as depicted in the following diagram:

\begin{center}
\includegraphics[scale=1]{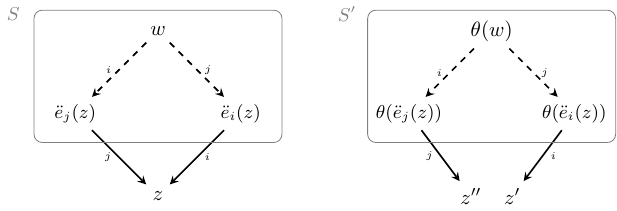} 
\end{center}

We will show that $z'$ does not depend on the choice of $i$, that is, if there exists $z'' \in \Q'$ such that $\qe_j (z'') = \theta (\qe_j (z))$, with $j \neq i$, then $z' = z''$. Since $\qe_i (z)$ and $\qe_j (z)$ are both defined, axiom \textbf{LQ3} implies that there exists $w \in S$ such that 
$$w = \qe_i \qe_j (z) = \qe_j \qe_i (z).$$
Therefore, we have $\theta (w) = \theta \qe_i \qe_j (z) = \theta \qe_j \qe_i (z)$, and by definition of $S$, $\theta(w) = \qe_i \theta \qe_j (z) = \qe_j \theta \qe_i (z)$. 
Thus, we have $\qf_i (\theta(w)) = \theta (\qe_j (w))$ and $\qf_j (\theta(w)) = \theta (\qe_i (w))$. It then follows from axiom \textbf{LQ3$'$} that
$$z' = \qf_i \qf_j (\theta (w)) = \qf_j \qf_i (\theta (w)) = z''.$$
Then, if $Q \neq S$, the map $\theta$ may be extend by defining $\theta (z) := z'$. This contradicts the maximality of $S$. Therefore, $Q = S$ and the map $\theta$ is a weight-preserving isomorphism.
\end{proof}

\begin{prop}
Let $\Q$ be a connected component of a seminormal quasi-crystal graph satisfying the local axioms, with highest weight element $u$. Then, $u$ has degree at most one.
\end{prop}

\begin{proof}
If $\Q$ has only one vertex, then the result is trivial. Suppose that $\Q$ has at least two vertices. Therefore, there exists $k \in I$ such that $\qf_k (u) \neq \bot$, and consequently, $\qphi_k (u), \qepsilon_k (u) \neq + \infty$, by \textbf{Q4}. Since $u$ is the highest weight element, we have $\qe_k (u) = \bot$, and since $\Q$ is seminormal, this implies that $\qepsilon_k (u)=0$.

We claim that $\qphi_j (u) = 0$, for $j \geq k$. If $k=n-1$, the result is trivial. Otherwise, axiom \textbf{LQ1} implies that $\qphi_{k+1} (u) = 0$, and, in particular, $\qphi_{k+1} (u) \neq + \infty$. Thus, by \textbf{Q2}, $\qepsilon_{k+1} (u) \neq + \infty$, and since $u$ is the highest weight element and $\Q$ is seminormal, we get $\qepsilon_{k+1} (u) = 0$. Then, by axiom \textbf{LQ1} we obtain $\qphi_{k+2} (u) = 0$, and applying the same reasoning, we have $\qphi_j (u) =0$, for $j \geq k$. Therefore, $\qf_j (u) = \bot$, for $j \geq k$.

If $k = 1$, the proof is done. Thus, suppose that $k > 1$, and we will prove that $\qf_j (u) = \bot$, for $j < k$.
Since $\qphi_k (u) \neq 0$ (because $\qf_k (u) \neq \bot$), by axiom \textbf{LQ1} we have $\qepsilon_{k-1} (u) \neq 0$. Then, as $\qe_{k-1} (u) = \bot$ and $\Q$ is seminormal, we must have $\qepsilon_{k-1} (u) = + \infty$. Consequently, by \textbf{Q4}, $\qf_{k-1} (u) = \bot$. 
Now suppose that $\qf_{k-2} (u)$ is defined. Then, since $\qepsilon_{k-1} (u) = + \infty$, by Corollary \ref{cor:infs}, there exists $l >0$ such that $\qepsilon_{k-1} (\qe_{k-2}^l (u)) \neq +\infty$. Since $l > 0$, this also contradicts the fact that $u$ is the highest weight element. Therefore, $\qf_{k-2} (u)$ is undefined. Applying the same reasoning, we conclude that $\qf_{j} (u)$ is undefined, for $j < k$. Thus, there is exactly one edge coming from $u$, and thus $u$ has degree one.
\end{proof}

\subsection{Quasi-tensor products}

We recall the notion of quasi-tensor product introduced by Cain, Guilherme and Malheiro \cite{CGM23}. Throughout this section, we will only consider seminormal quasi-crystals.

\begin{defin}[{\cite[Theorem 5.1]{CGM23}}]\label{def:qtens}
Let $\Q$ and $\Q'$ be seminormal quasi-crystals of type $A_{n-1}$. The quasi-tensor product $\Q \qtens \Q'$ is the Cartesian product $\Q \times \Q'$ together with the maps defined by:
\begin{enumerate}
\item $\wt(x \qtens x') = \wt(x) + \wt(x')$, for $x \in \Q$, $x' \in \Q'$.
\item If $\qphi_i(x) >0$ and $\qepsilon_i(x')>0$, for $i \in I$,
$$
\qe_i(x \qtens x') = \qf_i (x \qtens x') = \bot \qquad \text{and} \qquad
\qepsilon_i (x \qtens x') = \qphi_i (x \qtens x') = + \infty.
$$

\item Otherwise, 
$$\qe_i(x \qtens x') = 
\begin{cases}
\qe_i(x) \qtens x' & \text{if}\; \qphi_i(x) \geq \qepsilon_i(x')\\
x \qtens \qe_i(x') & \text{if}\; \qphi_i(x) < \qepsilon_i(x')
\end{cases}$$

$$\qf_i(x \qtens x') = 
\begin{cases}
\qf_i(x) \qtens x' & \text{if}\; \qphi_i(x) > \qepsilon_i(x')\\
x \qtens \qf_i(x') & \text{if}\; \qphi_i(x) \leq \qepsilon_i(x')
\end{cases}$$
and
\begin{align*}
\qepsilon_i(x) &= \max\{ \qepsilon_i(x), \qepsilon_i(x') - \langle \wt(x), \alpha_i \rangle \}\\
\qphi_i (x) &= \max \{ \qphi_i(x) + \langle \wt(x'), \alpha_i \rangle, \qphi_i(x') \}
\end{align*}
where $x \qtens \bot = \bot \qtens x' = \bot$.
\end{enumerate}
\end{defin}

The quasi-tensor product $\Q \qtens \Q'$ is a seminormal quasi-crystal of type $A_{n-1}$ \cite[Theorem 5.1]{CGM23}. We remark that, with this convention, $x \qtens y$ is identified with the word $yx$.
The following Lemma will be often used and is a direct consequence of \cite[Theorem 5.1]{CGM23}.

\begin{lem}\label{lem:tens_inf}
Let $x \in \Q$ and $x' \in \Q'$. If $\qepsilon_i (x \qtens x') = + \infty$, then $\qphi_{i}(x) > 0$ or $\qepsilon_{i}(x') > 0$.
\end{lem}

\begin{prop}[{\cite[Proposition 5.3]{CGM23}}]\label{prop:ricardo}
Let $\Q$ and $\Q'$ be seminormal quasi-crystals of the same type, and let $x \in \Q$ and $x' \in \Q'$ be such that $\qphi_i(x) = 0$ or $\qepsilon_i(x')=0$. Then:
$$
\qe_i(x \qtens x') = \begin{cases}
\qe_i(x) \qtens x' & \text{if } \qepsilon_i(x') =0\\
x \qtens \qe_i(x') & \text{if } \qepsilon_i(x') >0
\end{cases} \qquad
\qf_i(x \qtens x') = \begin{cases}
\qf_i(x) \qtens x' & \text{if } \qphi_i(x) > 0\\
x \qtens \qf_i(x') & \text{if } \qphi_i(x) = 0
\end{cases}
$$
and $\qepsilon_i(x \qtens x') = \qepsilon_i (x)+ \qepsilon_i (x')$, $\qphi_i (x \qtens x') = \qphi_i (x)+ \qphi_i (x')$.
\end{prop}

We may now state the following result, which is proved in the next section. 

\begin{thm}\label{thm:qtens}
Let $\Q$ and $\Q'$ be seminormal quasi-crystal graphs satisfying the local quasi-crystal axioms of Definition \ref{def:local_axioms}. Then, $\Q \qtens \Q'$ satisfies the same axioms.
\end{thm}

Recall that the \emph{standard crystal} $\mathcal{B}_n$ of type $A_{n-1}$, depicted in Figure \ref{fig:std_crystal}, is the crystal structure on $\{1 < \cdots < n\}$ in which $\cf_i (j)=j+1$ and $\cphi_i (j) = \delta_{i,j}$,  for $j \in I$.

\begin{figure}[h]
\begin{center}
\begin{tikzpicture}[scale=0.6]
\node (1) at (0,0) {$1$};
\node (2) at (2,0) {$2$};
\node (3) at (4,0) {$3$};
\node (4) at (6,0) {};
\node at (7,0) {$\cdots$};
\node (5) at (8,0) {}; 
\node (6) at (10,0) {$n$};

\draw (1) edge[-stealth, thick, red, above] node[black] {\tiny $1$} (2);
\draw (2) edge[-stealth, thick, blue, above] node[black] {\tiny $2$} (3);
\draw (3) edge[-stealth, thick, green!70!black, above] node[black] {\tiny $3$} (4);
\draw (5) edge[-stealth, thick, blue!50!green, above] node[black] {\tiny $n-1$} (6);
\end{tikzpicture}
\end{center}
\caption{The standard crystal $\mathcal{B}_n$ of type $A_{n-1}$.}
\label{fig:std_crystal}
\end{figure}
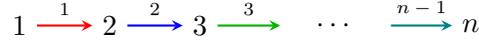

\begin{figure}
\begin{center}
\includegraphics[scale=0.8]{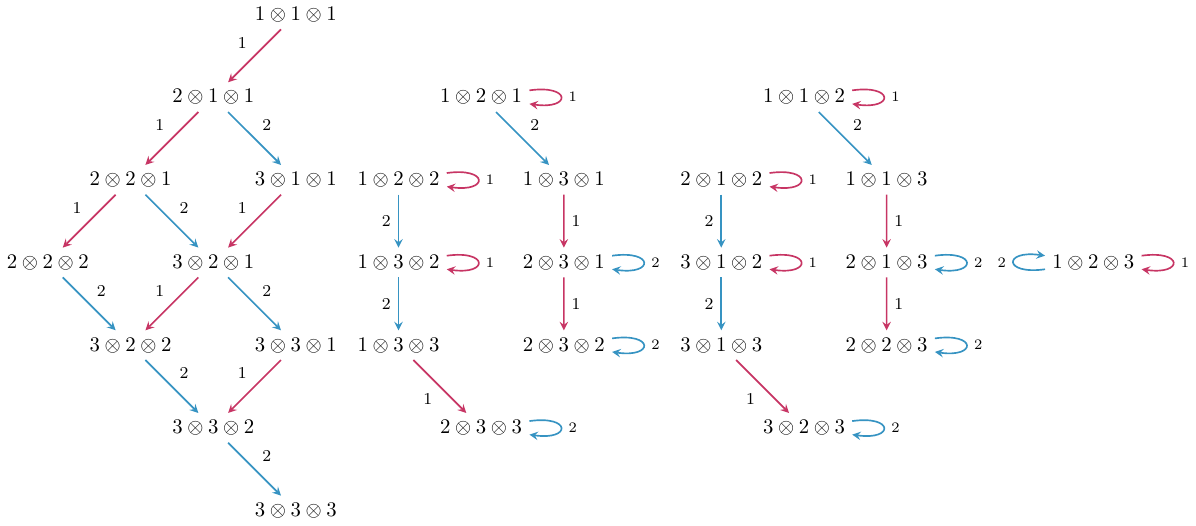}
\end{center}
\caption{The connected components of the quasi-tensor product $\mathcal{B}_3^{\qtens 3}$.}
\label{fig:quasi_tens}
\end{figure}

\begin{figure}
\begin{center}
\includegraphics[scale=0.8]{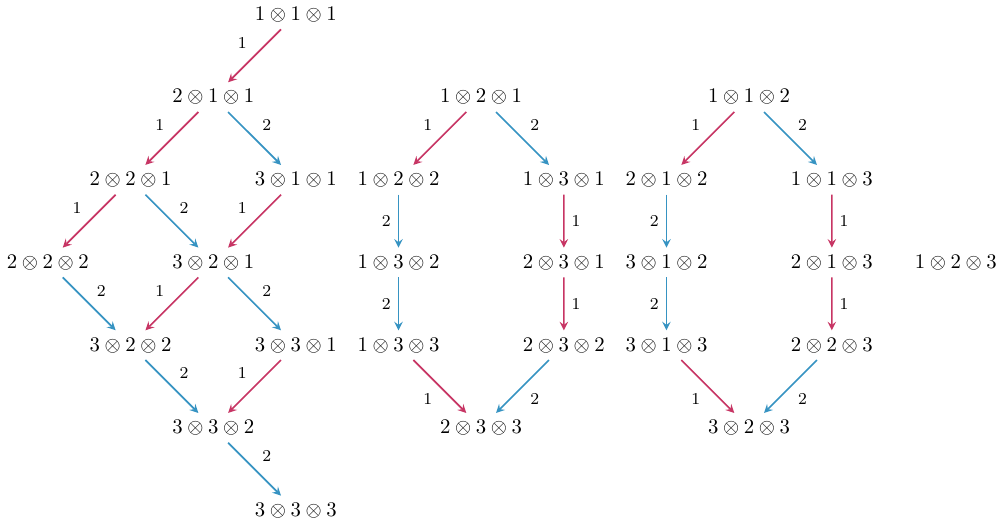}
\end{center}
\caption{The connected components of the usual tensor product $\mathcal{B}_3^{\otimes 3}$.}
\label{fig:crystal_tens}
\end{figure}

As a consequence of Theorem \ref{thm:qtens}, we have the following result. Figures \ref{fig:quasi_tens} and \ref{fig:crystal_tens} illustrate the differences between the usual tensor product of crystals and the quasi-tensor product, on the standard crystal $\mathcal{B}_3$ of type $A_2$.

\begin{cor}
Let $\mathcal{B}_n$ be the standard crystal of type $A_{n-1}$. Then, every connected component of $\mathcal{B}_n^{\qtens k}$ satisfies the local quasi-crystal axioms of Definition \ref{def:local_axioms}.
\end{cor}

\begin{proof}
The standard crystal $\mathcal{B}_n$ clearly satisfies the local quasi-crystal axioms. If follows from Theorem \ref{thm:qtens} that $\mathcal{B}_n^{\qtens k}$ satisfies the same axioms as well. 
\end{proof}

\subsection{Proof of Theorem \ref{thm:qtens}}

\begin{prop}\label{prop:LQ1_qtens}
$\Q \qtens \Q'$ satisfies axiom \emph{\textbf{LQ1}}.
\end{prop}

\begin{proof}
Suppose that $\qepsilon_i (x \qtens x') = 0$. Then, by Definition \ref{def:qtens}, we must have $\qphi_i (x) = 0$ or $\qepsilon_i (x')=0$. Proposition \ref{prop:ricardo} then implies that $\qepsilon_i (x \qtens x') = \qepsilon_i (x) + \qepsilon_i (x') =0$. Since $\Q$ and $\Q'$ are seminormal, we have $\qepsilon_i (x), \qepsilon_i (x') \geq 0$. Therefore, we must have $\qepsilon_i (x) = \qepsilon_i (x') = 0$, and then, axiom \textbf{LQ1} implies that $\qphi_{i+1} (x) = \qphi_{i+1} (x') = 0$. In particular, since $\qphi_{i+1}(x) =0$, by Proposition \ref{prop:ricardo}, we have $\qphi_{i+1} (x \qtens x') = \qphi_{i+1} (x) + \qphi_{i+1} (x') =0$. The reverse implication is proved similarly.
\end{proof}

\begin{prop}
$\Q \qtens \Q'$ satisfies \emph{\textbf{LQ2}}.
\end{prop}

\begin{proof}
Suppose that $\qe_i(x \qtens x') = y \qtens y'$. Then, by Definition \ref{def:qtens}, we have $\qphi_i (x) =0$ or $\qepsilon_i (x')=0$.

\noindent
\textbf{We show that $\Q \qtens \Q'$ satisfies LQ2} (1). Let $i, j \in I$ be such that $|i-j| > 1$. Then, by Proposition \ref{prop:ricardo}, we have two cases. We prove the case where $\qepsilon_i (x') = 0$. The case where $\qepsilon_i(x') > 0$ is proved similarly.
Suppose that $\qepsilon_i(x') = 0$. Then, $\qe_i(x \qtens x') = \qe_i (x) \qtens x' = y \qtens y'$, and we have
$$
\qe_i(x) = y, \quad x' = y'.
$$
If $\qphi_j (y), \qepsilon_j(y') > 0$, then, as $\qe_i(x) = y$, we have $\qphi_j (x) = \qphi_j (y)$ by Proposition \ref{prop:local_ax}, and since $x' = y'$, we have $\qepsilon_j (x') = \qepsilon_j (y')$.
Then, by Definition \ref{def:qtens}, we have $\qepsilon_j (y \qtens y') = \qepsilon_j(x \qtens x') = + \infty$.
Otherwise, suppose that $\qphi_j (y) = 0$ or $\qepsilon_j (y') = 0$. Then, by axiom \textbf{LQ1}, we have $\qepsilon_j (x) = \qepsilon_j (y)$, and since $x'=y'$, we have $\qepsilon_j (x') = \qepsilon_j (y')$. Thus, by Proposition \ref{prop:ricardo}, we have
$$
\qepsilon_j (y \qtens y') = \qepsilon_j (y) + \qepsilon_j (y') = \qepsilon_j (x) + \qepsilon_j (x') = \qepsilon_j(x \qtens x').
$$

\noindent
\textbf{We show that $\Q \qtens \Q'$ satisfies LQ2} (2). The proof for \textbf{LQ2} (3) is done similarly.

\begin{enumerate}
\item We first suppose that $\qepsilon_{i+1} (x \qtens x') \neq + \infty$, and we prove that $\qepsilon_{i+1} (x \qtens x') = \qepsilon_{i+1} (y \qtens y')$. By Definition \ref{def:qtens}, we have $\qphi_{i+1} (x) =0$ or $\qepsilon_{i+1} (x') = 0$, and by Proposition \ref{prop:ricardo}, we have $\qepsilon_{i+1} (x \qtens x') = \qepsilon_{i+1}(x) + \qepsilon_{i+1}(x')$. In particular, $\qepsilon_{i+1}(x)$ and $\qepsilon_{i+1}(x')$ must be finite. Moreover, since $\qe_i(x \qtens x') = y \qtens y'$, we have $\qphi_i(x) = 0$ or $\qepsilon_i (x')=0$. Following Proposition \ref{prop:ricardo}, we have two cases to consider:

\noindent
\textbf{Case 1.1.} Suppose that $\qepsilon_i (x') = 0$. Then,
$$
\qe_i(x) = y, \quad x' = y'.
$$
Since $\qepsilon_{i+1}(x)$ is finite, Proposition \ref{prop:local_ax} implies that
\begin{equation}\label{eq:2_2_eps_phs}
\qepsilon_{i+1}(y) = \qepsilon_{i+1}(x), \quad \qphi_{i+1}(y) = \qphi_{i+1}(x)-1.
\end{equation}
Therefore, since $\Q$ is seminormal, we have $\qphi_{i+1}(x)>0$, and consequently, $\qepsilon_{i+1} (x') = 0$. Since $x' = y'$, we have
\begin{equation}\label{eq:2_2_zero}
\qepsilon_{i+1} (x') = \qepsilon_{i+1}(y') = 0.
\end{equation}
Then, by Proposition \ref{prop:ricardo}, and equations \eqref{eq:2_2_eps_phs} and \eqref{eq:2_2_zero}, we have
\begin{align*}
\qepsilon_{i+1}(y \qtens y') &= \qepsilon_{i+1} (y) + \qepsilon_{i+1}(y') \\
&=\qepsilon_{i+1} (x) + \qepsilon_{i+1}(x') \\
&=\qepsilon_{i+1} (x \qtens x')
\end{align*}

\noindent
\textbf{Case 1.2.} Now suppose that $\qepsilon_i(x') > 0$. Then, we have $\qphi_i(x) = 0$ and 
$$
x =y, \quad \qe_i(x')=y'.
$$
Since $\qepsilon_{i+1}(x')$ is finite, by axiom \textbf{LQ2} (2) we have
\begin{equation}\label{eq:2_2_eps_p}
\qepsilon_{i+1}(x') = \qepsilon_{i+1}(y').
\end{equation}
If $\qphi_{i+1}(x)=0$, since $x=y$, we get $\qphi_{i+1}(y)=0$. Otherwise, we must have $\qepsilon_{i+1}(x')=0$, and by \eqref{eq:2_2_eps_p}, we have $\qepsilon_{i+1}(y') = 0$. Thus, we have $\qphi_{i+1}(y)=0$ or $\qepsilon_{i+1}(y') = 0$, and by Proposition \ref{prop:ricardo},
\begin{align*}
\qepsilon_{i+1}(y \qtens y') &= \qepsilon_{i+1} (y) + \qepsilon_{i+1}(y')\\
&=\qepsilon_{i+1} (x) + \qepsilon_{i+1}(x') \\
&=\qepsilon_{i+1} (x \qtens x')
\end{align*}

\item Now suppose that $\qepsilon_{i+1} (x \qtens x') = +\infty$. By Lemma \ref{lem:tens_inf}, we have $\qphi_{i+1}(x)>0$ or $\qepsilon_{i+1}(x')>0$.
We will show that $\qepsilon_{i+1}(x \qtens x') \neq \qepsilon_{i+1}(y \qtens y')$ if, and only if, $\qepsilon_{i} (y \qtens y') = 0$.

Suppose that $\qepsilon_i(y \qtens y') = 0$. By Proposition \ref{prop:LQ1_qtens}, we have $\qphi_{i+1} (y \qtens y') = 0 \neq + \infty$. Therefore, by \textbf{Q2}, $\qepsilon_{i+1} (y \qtens y') \neq +\infty$ and hence $\qepsilon_{i+1}(y \qtens y') \neq \qepsilon_{i+1}(x \qtens x') = + \infty$.

Now suppose that $\qepsilon_{i+1}(x \qtens x') \neq \qepsilon_{i+1}(y \qtens y')$. Then, $\qepsilon_{i+1}(y \qtens y') \neq + \infty$, which implies, in particular, that $\qepsilon_{i+1}(y)$ and $\qepsilon_{i+1}(y')$ are finite, and thus, $\qphi_{i+1}(y) = 0$ or $\qepsilon_{i+1}(y') = 0$. By Proposition \ref{prop:ricardo}, we have two cases to consider.

\noindent
\textbf{Case 2.1.} Suppose that $\qepsilon_{i}(x') = 0$. Then,
$$
\qe_i(x)=y, \quad x'=y'.
$$
This implies that $\qepsilon_i(y') = \qepsilon_i(x') =0$, and thus, by Proposition \ref{prop:ricardo}, we have
\begin{equation}\label{eq:2_2_y1}
\qepsilon_i(y \qtens y') = \qepsilon_i (y)+\qepsilon_i(y') = \qepsilon_i (y).
\end{equation}
If $\qphi_{i+1}(y)=0$, then axiom \textbf{LQ1} implies that $\qepsilon_i (y) = 0$. Otherwise, if $\qphi_{i+1}(y) >0$, we must have $\qepsilon_{i+1} (y') = 0$. Since $x'=y'$, we have $\qepsilon_{i+1}(x') = 0$, and by Proposition \ref{prop:ricardo},
$$\qepsilon_{i+1} (x \qtens x') = \qepsilon_{i+1}(x) + \qepsilon_{i+1} (x') = \qepsilon_{i+1} (x) = +\infty.$$
Since $\qepsilon_{i+1}(y) \neq + \infty$, we have $\qepsilon_{i+1}(x) \neq \qepsilon_{i+1}(y)$, and then, axiom \textbf{LQ2} (2) implies that $\qepsilon_i(y)=0$. In either case we get $\qepsilon_i(y) = 0$, and thus, by \eqref{eq:2_2_y1}, $\qepsilon_{i}(y \qtens y') = 0$.

\noindent
\textbf{Case 2.2.} Now suppose that $\qepsilon_i(x') > 0$. Then, we have $\qphi_i(x) = 0$ and
$$
x=y, \quad \qe_i(x')=y'.
$$
Since $\qepsilon_{i+1}(y \qtens y') \neq +\infty$, we have $\qphi_{i+1}(y) = 0$ or $\qepsilon_{i+1}(y')=0$. We claim that $\qepsilon_{i+1}(y') > 0$. If $\qepsilon_{i+1}(y') = 0$, then, as $\qe_i (x') = y'$, axiom \textbf{LQ2} (2) implies that $\qepsilon_{i+1}(x') = \qepsilon_{i+1}(y') = 0$, and by Proposition \ref{prop:ricardo},
\[
\qepsilon_{i+1} (x) = \qepsilon_{i+1}(x) + \qepsilon_{i+1}(x') = \qepsilon_{i+1}(x \qtens x') = + \infty.
\]
But since $x=y$, we have $\qepsilon_{i+1}(y) = + \infty$, and thus $\qepsilon_{i+1}(y \qtens y') = + \infty$, which contradicts the hypothesis that $\qepsilon_{i+1}(y \qtens y') \neq + \infty$. Thus, we have $\qepsilon_{i+1}(x') > 0$ and consequently, $\qphi_{i+1}(y) = 0$. Axiom \textbf{LQ1} then implies that $\qepsilon_i(y) = 0$. Since $\qphi_i(x) = 0$ and $x = y$, we get $\qphi_i (y) = 0$. Thus, by Proposition \ref{prop:ricardo}, we have
\begin{equation}\label{eq:2_2_y2}
\qepsilon_{i}(y \qtens y') = \qepsilon_{i}(y) + \qepsilon_{i}(y') = \qepsilon_{i} (y').
\end{equation} 
Moreover, since $\qphi_{i+1}(y) = 0$ and $x = y$, we have $\qphi_{i+1}(x) = 0$, and by Proposition \ref{prop:ricardo},
\begin{equation}\label{eq:2_2_x_inf}
\qepsilon_{i+1} (x) + \qepsilon_{i+1}(x') = \qepsilon_{i+1} (x \qtens x') = +\infty,
\end{equation}
which implies that $\qepsilon_{i+1}(x) = + \infty$ or $\qepsilon_{i+1} (x') = + \infty$. Since $\qphi_{i+1}(x) = 0$ (in particular, it is finite), we get that $\qepsilon_{i+1}(x)$ is finite. Thus, \eqref{eq:2_2_x_inf} implies that $\qepsilon_{i+1}(x') = + \infty$. Since $\qepsilon_{i+1}(y') \neq + \infty$, we have $\qepsilon_{i+1} (x') \neq \qepsilon_{i+1}(y')$, and thus, by axiom \textbf{LQ2} (2), we have $\qepsilon_i(y') = 0$. Therefore, by \eqref{eq:2_2_y2}, we have $\qepsilon_{i}(y \qtens y') = 0$.

\item Finally, we show that $\qepsilon_{i+1}(y \qtens y') = 0$ implies that $\qepsilon_{i+1}(x \qtens x') = \qepsilon_{i+1}(y \qtens y')$. If $\qepsilon_{i+1}(y \qtens y') = 0$, then, in particular, $\qepsilon_{i+1}(y \qtens y')$ is finite, and thus $\qphi_{i+1}(y)=0$ or $\qepsilon_{i+1}(y')=0$. Thus, by Proposition \ref{prop:ricardo} we have
\[
\qepsilon_{i+1}(y) + \qepsilon_{i+1}(y') = \qepsilon_{i+1} (y \qtens y') = 0.
\]
This implies, since $\Q$ and $\Q'$ are seminormal, that $\qepsilon_{i+1} (y) = \qepsilon_{i+1} (y') = 0$. If $\qepsilon_i(x')=0$, we get $\qe_i(x)=y$ and $\quad x'=y'$.
Since $\qepsilon_{i+1}(y) = 0$, axiom \textbf{LQ2} (2) implies that $\qepsilon_{i+1}(x) = \qepsilon_{i+1}(y) = 0$. And since $x' = y'$, we have $\qepsilon_{i+1}(x') = \qepsilon_{i+1}(y') = 0$. Therefore, by Proposition \ref{prop:ricardo}, we have
\[
\qepsilon_{i+1} (x \qtens x') = \qepsilon_{i+1} (x) + \qepsilon_{i+1} (x') = \qepsilon_{i+1} (y) + \qepsilon_{i+1} (y') = \qepsilon_{i+1} (y \qtens y').
\]
If $\qepsilon_i(x')>0$, then $x=y$ and $\qe_i(x') = y'$. Since $\qepsilon_{i+1}(y')=0$, as before, axiom \textbf{LQ2} (2) implies that $\qepsilon_{i+1} (x') = \qepsilon_{i+1} (y') = 0$, and as $x=y$, we get $\qepsilon_{i+1} (x)=\qepsilon_{i+1}(y)=0$. Reasoning as before, we conclude that $\qepsilon_{i+1} (x \qtens x') = \qepsilon_{i+1}(y \qtens y')$.
\qedhere
\end{enumerate}
\end{proof}

\begin{prop}
$\Q \qtens \Q'$ satisfies \emph{\textbf{LQ3}} and \emph{\textbf{LQ3$'$}}.
\end{prop}

\begin{proof}
\textbf{We show that $\Q \qtens \Q'$ satisfies LQ3.} Let $i, j \in I$, with $i \neq j$, and suppose that $\qe_i(x \qtens x')$ and $\qe_j(x \qtens x')$ are both defined. 
By Proposition \ref{prop:ricardo}, we have four cases to consider.

\noindent
\textbf{Case 1.} Suppose that $\qepsilon_i(x') = 0$ and $\qepsilon_j(x') = 0$. Then, by Proposition \ref{prop:ricardo}, we have 
\begin{equation}\label{eq:3_1_qt}
\qe_i(x \qtens x') = \qe_i (x) \qtens x', \quad \qe_j(x \qtens x') = \qe_j (x) \qtens x'.
\end{equation}

Since $\qe_i(x) = y$ and $\qe_j(x)= z$ are both defined, by axiom \textbf{LQ3} we have
\begin{equation}\label{eq:3_1_LQ}
\qe_i \qe_j (x) = \qe_j \qe_i (x) \neq \bot.
\end{equation}

Then, since $\qepsilon_i (x') = \qepsilon_j(x') = 0$, we have
\begin{align*}
\qe_i \qe_j (x \qtens x') &= \qe_i (\qe_j(x) \qtens x') &\text{(by \eqref{eq:3_1_qt})}\\
&= \qe_i \qe_j (x) \qtens x' &\text{(by Proposition \ref{prop:ricardo})}\\
&= \qe_j \qe_i (x) \qtens x' &\text{(by \eqref{eq:3_1_LQ})}\\
&= \qe_j (\qe_i(x) \qtens x') &\text{(by Proposition \ref{prop:ricardo})}\\
&= \qe_j \qe_i (x \qtens x'). &\text{(by \eqref{eq:3_1_qt})}
\end{align*}

\noindent
\textbf{Case 2.} Suppose that $\qepsilon_i (x') =0$ and $\qepsilon_j(x') > 0$. Then, we have $\qphi_j(x)=0$ and
\begin{equation}\label{eq:3_2_qt}
\qe_i (x \qtens x') = \qe_i (x) \qtens x', \quad \qe_j(x \qtens x') = x \qtens \qe_j(x').
\end{equation}
We claim that $j \neq i+1$. If $j = i + 1$, then, since $\qe_i(x)$ is defined, we have $\qepsilon_i(x)>0$, and axiom \textbf{LQ1} implies that $\qphi_{i+1}(x)>0$. And since $\qe_j (x') = \qe_{i+1} (x')$ is defined, we have $\qepsilon_{i+1} (x')>0$. Thus, by Definition \ref{def:qtens}, we have $\qe_j(x \qtens x') = \qe_{i+1} (x \qtens x')=\bot$, which contradicts the hypothesis. Thus, $j \neq i+1$, and we have two cases to consider.
\begin{description}
\item[Case 2.1] Suppose that $|i-j|>1$. By axiom \textbf{LQ2} (1), we have $\qepsilon_i(\qe_j (x')) = \qepsilon_i (x') = 0$. Thus, by Proposition \ref{prop:ricardo},
\begin{equation}\label{eq:3_2_ij}
\qe_i \qe_j (x \qtens x') = \qe_i (x \qtens \qe_j(x')) = \qe_i (x) \qtens \qe_j (x'). 
\end{equation}
Axiom \textbf{LQ2} (1) also implies that $\qepsilon_j (\qe_i (x)) = \qepsilon_j (x) > 0$, and therefore, by Proposition \ref{prop:ricardo},
\begin{equation}\label{eq:3_2_ji}
\qe_j \qe_i (x \qtens x') = \qe_j (\qe_i (x) \qtens x') = \qe_i (x) \qtens \qe_j (x').
\end{equation}
From \eqref{eq:3_2_ij} and \eqref{eq:3_2_ji}, we get $\qe_i \qe_j (x \qtens x') = \qe_j \qe_i (x \qtens x')$.

\item[Case 2.2] Now suppose that $j=i-1$.
We claim that $\qphi_{i-1}(\qe_i (x))$ is finite. If $\qphi_{i-1} (\qe_i(x)) = + \infty$, then, as $\qphi_{i-1} (x) = \qphi_j (x) = 0$, we would have $\qphi_{i-1} (\qe_i (x)) \neq \qphi_{i-1} (x)$. Thus, by axiom \textbf{LQ2} (3), we would have $\qphi_{i-1}(x) > 0$, which is a contradiction. Therefore, we have $\qphi_{i-1}(\qe_i(x)) \neq + \infty$, and then, axiom \textbf{LQ2} (3) implies that $\qphi_{i-1} (\qe_i (x)) = \qphi_{i-1} (x) = 0$. Consequently, $\qe_{i-1} (\qe_i (x) \qtens x')$ is defined. Moreover, $\qepsilon_j (x') = \qepsilon_{i-1}(x') > 0$, and thus Proposition \ref{prop:ricardo} implies that
\begin{equation}\label{eq:3_2_i1_i}
\qe_{i-1} \qe_i (x \qtens x') = \qe_{i-1} (\qe_i (x) \qtens x') = \qe_i (x) \qtens \qe_{i-1} (x'). 
\end{equation}
Since $\qepsilon_i(x') = 0$, in particular, $\qepsilon_i (x') \neq + \infty$. Thus, by axiom \textbf{LQ2} (2), we have $\qepsilon_i (\qe_{i-1}(x')) = \qepsilon_i (x') = 0$, and by Proposition \ref{prop:ricardo},
\begin{equation}\label{eq:3_2_i_i1}
\qe_i \qe_{i-1} (x \qtens x') = \qe_i (x \qtens \qe_{i-1} (x')) = \qe_i (x) \qtens \qe_{i-1} (x'). 
\end{equation}
From \eqref{eq:3_2_i1_i} and \eqref{eq:3_2_i_i1}, we get $\qe_{i-1} \qe_i (x \qtens x') = \qe_i \qe_{i-1} (x \qtens x').$
\end{description}

\noindent
\textbf{Case 3.} Suppose that $\qepsilon_i (x') > 0$ and $\qepsilon_j (x') = 0$. This is similar to the previous case, except now we have $j \neq i-1$.

\noindent
\textbf{Case 4.} Suppose that $\qepsilon_i (x') > 0$ and $\qepsilon_j(x')>0$. Then, we have $\qphi_i (x) = \qphi_j (x) = 0$, and
\begin{equation}\label{eq:3_4_qt}
\qe_i (x \qtens x') = x \qtens \qe_i(x'), \quad \qe_j (x \qtens x') = x \qtens \qe_j(x').
\end{equation}
From \eqref{eq:3_4_qt}, we get that $\qepsilon_i (\qe_j (x')) >0$ and $\qepsilon_j (\qe_i (x'))>0$. Since $\qe_i (x')$ and $\qe_j (x')$ are both defined, we have $\qe_i \qe_j (x') = \qe_j \qe_i (x') \neq \bot$, by axiom \textbf{LQ3}. And since $\qphi_i (x) = \qphi_j (x) = 0$, by Proposition \ref{prop:ricardo}, we have
\begin{align*}
\qe_i \qe_j (x \qtens x') &= \qe_i (x \qtens \qe_j (x'))\\
&= x \qtens \qe_i \qe_j (x')\\
&= x \qtens \qe_j \qe_i (x')\\
&= \qe_j (x \qtens \qe_i (x'))\\
&= \qe_j \qe_i (x \qtens x'). \qedhere
\end{align*}
\end{proof}

\section{Stembridge crystals and quasi-crystals}\label{sec:stem_crystals_quasi}

In what follows we consider $\C$ to be a crystal graph of type $A_{n-1}$.

\begin{obs}\label{rmk:wt_lattice}
We consider the weight lattice of $\C$ to be $\mathbb{Z}^n$ modulo $\mathbf{e_1} + \cdots + \mathbf{e_n}$. Therefore, we may fix a representative such that highest weights are always partitions, followed by a possible sequence of zeros. This implies that the entries of $\wt(x)$ are non-negative, for any $x \in \C$.
\end{obs}

We now state some auxiliary results.
\begin{lem}\label{lem:ax_s1}
Let $\C$ be a Stembridge crystal, and suppose that $\ce_i(x)=y \in \C$. Let $j\neq i$. Then,
\begin{enumerate}
\item $\cepsilon_j (y) \geq \cepsilon_j (x)$,
\item $\cphi_j (y) \leq \cphi_j (x)$.
\end{enumerate}
\end{lem}

\begin{proof}
This is a direct consequence of Stembridge axioms.
\end{proof}

\begin{lem}\label{lem:leq_wt}
Let $\C$ be a seminormal crystal. Then,
\begin{enumerate}
\item $\cepsilon_i (x) \leq \wt_{i+1} (x)$,
\item $\cphi_i (x) \leq \wt_i (x)$.
\end{enumerate}
\end{lem}

\begin{proof}
We prove the first statement; the second is similar. Since $\C$ is seminormal, let $k_0 := \cepsilon_i (x) = \max \{k: \ce_i^{\,k} (x) \neq \bot\}$. If $k_0 > \wt_{i+1} (x)$, then, as $\wt (\ce_i^{\,k_0} (x)) = \wt (x) + k_0 \alpha_i$, we would have $\wt_{i+1} (\ce_i^{\,k_0} (x)) = \wt_{i+1} (x)-k_0 < 0$, which is a contradiction by Remark \ref{rmk:wt_lattice}.
\end{proof}

\begin{lem}\label{lem:eps_phis_eq}
$\cepsilon_i (x) = \wt_{i+1} (x)$ if and only if $\cphi_i (x) = \wt_i (x)$.
\end{lem}

\begin{proof}
Suppose that $\cepsilon_i(x) = \wt_{i+1} (x)$. Then, by \textbf{C2}, $\cphi_i (x) = \cepsilon_i (x) + \langle \wt(x), \alpha_i \rangle = \wt_{i+1} (x) + \wt_i (x) - \wt_{i+1} (x) = \wt_i (x)$. The other implication is proved analogously.
\end{proof}

Next, we provide a construction to obtain quasi-crystal graphs from a connected crystal graph. Recently, Maas-Gariépy introduced an equivalent construction, by considering the induced subgraphs corresponding to fundamental quasi-symmetric functions \cite{MG23}.

\begin{defin}\label{def:crystal_quasi}
Let $(\C, \ce_i, \cf_i, \cepsilon_i, \cphi_i)$ be a connected Stembridge crystal. We construct a quasi-crystal $(\Q_{\C}, \qe_i, \qf_i, \qepsilon_i, \qphi_i)$ as follows: $\Q_{\C}$ has the same underlying set, index set $I$ and weight function $\wt$ as $\C$, and we define
\[
\qepsilon_i (x) := \begin{cases}
\cepsilon_i (x) & \text{if}\; \cepsilon_i (x) = \wt_{i+1} (x)\\
+ \infty & \text{otherwise}
\end{cases},\]
\[
\qe_i (x) := \begin{cases}
\ce_i (x) & \text{if}\; \cepsilon_i (x) = \wt_{i+1} (x)\\
\bot & \text{otherwise}
\end{cases}\]
and set $\qphi_i (x) := \qepsilon_i (x) + \langle \wt(x), \alpha_i \rangle$, and
 $\qf_i (x) := y$ if and only if $\qe_i (y)=x$.
\end{defin}

If $\cf_i (x) \neq \bot$ and $\qf_i (x) = \bot$, we say that $\cf_i$ is a \emph{strict Kashiwara operator} on $x \in \C$. Otherwise, it is simply called a quasi-Kashiwara operator.

In what follows we consider $\C$ and $\Q_{\C}$ as in Definition \ref{def:crystal_quasi}.

\begin{lem}\label{lem:inf_both}
Let $x, y \in \C$ be such that $\ce_i (x) = y$. Then, using the notation of Definition \ref{def:crystal_quasi}, $\qepsilon_i (x) = + \infty$ if and only if $\qepsilon_i (y) = + \infty$.
\end{lem}

\begin{proof}
By Definition \ref{def:crystal_quasi} and Lemma \ref{lem:leq_wt}, $\qepsilon_i (x) = + \infty$ implies that $\cepsilon_i (x) < \wt_{i+1} (x)$, and consequently 
$$\cepsilon_i (x) -1 < \wt_{i+1} (x) -1.$$ 
Since $\ce_i (x) = y$, by \textbf{C1} we have $\cepsilon_i (y) = \cepsilon_i (x)-1$ and $\wt_{i+1} (y) = \wt_{i+1}(x)-1$, and thus $\cepsilon_i (y) < \wt_{i+1} (y)$ and $\qepsilon_i(y) = + \infty$. The proof of the reverse implication is analogous.
\end{proof}

\begin{lem}
Let $i, j \in I$ be such that $|i-j|>1$ and suppose that $\qe_i(x) = y$. Then, $\qepsilon_j(x) = + \infty$ if and only if $\qepsilon_j (y) = + \infty$.
\end{lem}

\begin{proof}
Since $\qe_i(x) = y$, we have $\ce_i(x) = y$ (and $\cepsilon_i (x) = \wt_{i+1}(x)$), and by axiom \textbf{S1}, $\cepsilon_j (x) = \cepsilon_j (y)$. Suppose that $\qepsilon_j (x) = + \infty$. Then, $\cepsilon_j (x) < \wt_{j+1} (x)$ and thus, $\cepsilon_j (y) = \cepsilon_j (x) < \wt_{i+1} (x)$. Since $|i-j|>1$, we have $\wt_{j+1} (x) = \wt_{j+1} (y)$. Therefore, $\cepsilon_j(y) < \wt_{j+1}(y)$ and $\qepsilon_j (y) = + \infty$. The other implication is proved similarly.
\end{proof}

\begin{lem}
Suppose that $\qepsilon_i(x) = + \infty$. Then, $\wt_i (x), \wt_{i+1}(x) >0$.
\end{lem}

\begin{proof}
Since $\qepsilon_i (x) = + \infty$, we have $\cepsilon_i (x) < \wt_{i+1} (x)$ and $\cphi_i(x) < \wt_i (x)$, by Definition \ref{def:crystal_quasi} and Lemma \ref{lem:eps_phis_eq}. Since $\C$ is seminormal, we have $\cepsilon_i (x),\, \cphi_i (x) \geq 0$, which implies that $\wt_i (x),\, \wt_{i+1}(x) >0$.
\end{proof}

We prove the following result in subsection \ref{subsec:proof_thm} below.

\begin{thm}\label{thm:crystal_quasi_ax}
The quasi-crystal graph $\Q_{\C}$ is seminormal and satisfies the axioms of Definition \ref{def:local_axioms}.
\end{thm}

An example of the construction of Definition \ref{def:crystal_quasi} is depicted in Figure \ref{fig:crystal_quasi_3121}.

\begin{figure}[h]
\includegraphics[scale=0.9]{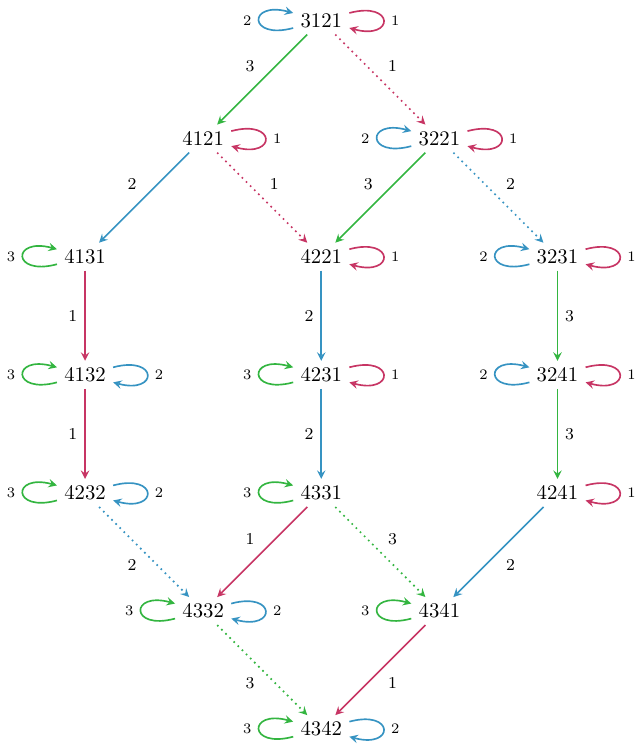}
\caption{The crystal graph having highest weight element $3121$, partitioned into quasi-crystal components having highest weight elements $3121$, $3221$ and $3231$. The solid non-loop edges denote the action of quasi-Kashiwara operators and the dashed ones denote the action of strict Kashiwara operators. The loops were added according to Definition \ref{def:crystal_quasi}.}
\label{fig:crystal_quasi_3121}
\end{figure}

\begin{cor}
Let $\C$ be a connected Stembridge crystal having highest weight $\lambda$. Then the number of connected components in $\Q_{\C}$ is given by $f_{\lambda}$, the number of standard Young tableaux of shape $\lambda$.
\end{cor}

\begin{proof}
The character of $\C$ is the Schur function $s_{\lambda}$, where $\lambda$ is the highest weight (see, for instance, \cite{BumpSchi17}). Since the character of a quasi-crystal connected component is a fundamental quasi-symmetric function $F_{\alpha}$, taking the characters of $\Q_{\C}$, one obtains a decomposition of $s_{\lambda}$ as a sum of fundamental quasi-symmetric functions. The result then follows from the following decomposition \cite{Gess19}:
 $$s_{\lambda} = \sum\limits_{T \in \mathsf{SYT}(\lambda)} F_{\mathsf{DesComp}(T)},$$ 
where $\mathsf{DesComp}(T)$ denotes the descent composition of a standard tableau $T$, together with the fact that the fundamental quasi-symmetric functions $F_{\alpha}$ form a basis for the ring of quasi-symmetric functions.
\end{proof}

\subsection{Proof of Theorem \ref{thm:crystal_quasi_ax}}\label{subsec:proof_thm}

In this section we prove Theorem \ref{thm:crystal_quasi_ax}. For simplicity, we let $\Q := \Q_{\C}$. By construction, $\Q$ is seminormal.

\begin{prop}\label{prop:q_LQ1}
$\Q$ satisfies axiom \emph{\textbf{LQ1}}.
\end{prop}

\begin{proof}
Let $i, i+1 \in I$. Suppose that $\qepsilon_i (x) = 0$. Since $\qepsilon_i (x) \neq + \infty$, by Definition \ref{def:crystal_quasi}, we have $\qepsilon_i (x) = \cepsilon_i (x) = \wt_{i+1} (x) = 0$. By Lemma \ref{lem:leq_wt}, we have $\cphi_{i+1} (x) \leq \wt_{i+1} (x) = 0$. Therefore, $\cphi_{i+1}(x) = 0 = \wt_{i+1} (x)$ and thus $\qphi_{i+1} (x) = 0$. The other implication is proved analogously. 
\end{proof}

\begin{prop}\label{prop:q_LQ2}
$\Q$ satisfies axiom \emph{\textbf{LQ2}}.
\end{prop}

\begin{proof}
Suppose that $\qe_i (x) = y$. By Definition \ref{def:crystal_quasi}, we have $\ce_i (x) = y$.

\medskip
\textbf{We show that $\Q$ satisfies LQ2} (1). Suppose that $|i-j|>1$. We claim that $\qepsilon_j (x) = + \infty$ if and only if $\qepsilon_j (y) = + \infty$. Indeed, since $\ce_i(x) = y$, by \textbf{C1} we have $\wt_{j+1} (x) = \wt_{j+1} (y)$, and by  axiom \textbf{S1}, we have $\cepsilon_j(x) = \cepsilon_j (y)$. Therefore, $\cepsilon_j (x) < \wt_{j+1} (x)$ if and only if $\cepsilon_j (y) < \wt_{j+1} (y)$.
Thus, if $\qepsilon_j (x) = + \infty$, then $\qepsilon_j (y) = + \infty = \qepsilon_j (x)$. If $\qepsilon_j (x) \neq + \infty$, then $\qepsilon_j (y) \neq + \infty$, therefore $\qepsilon_j (x) = \cepsilon_j (x)$ and $\qepsilon_j (y) = \cepsilon_j (y)$. By axiom \textbf{S1}, we have $\cepsilon_j (x) = \cepsilon_j (y)$, which concludes the proof of \textbf{LQ2} (1).

\textbf{We show that $\Q$ satisfies LQ2} (2). The proof of \textbf{LQ2} (3) is analogous. Let $i+1 \in I$

\begin{enumerate}
\item Suppose that $\qepsilon_{i+1} (x) = + \infty$ and that $\qepsilon_i (y) = 0$. By Proposition \ref{prop:q_LQ1}, $\qepsilon_i (y)=0$ implies that $\qphi_{i+1} (y) = 0$. In particular, $\qphi_{i+1} (y) \neq + \infty$, and by Definition \ref{def:crystal_quasi}, we have $\qepsilon_{i+1} (y) \neq + \infty$. Therefore, $\qepsilon_{i+1} (x) = + \infty \neq \qepsilon_{i+1} (y)$.

\item Now suppose that $\qepsilon_{i+1} (y) = 0$. We will show that $\qepsilon_{i+1} (x) = \qepsilon_{i+1} (y)$. By Definition \ref{def:crystal_quasi} and Lemma \ref{lem:eps_phis_eq}, as $\qepsilon_{i+1} (y) = 0 \neq + \infty$, we have $\cepsilon_{i+1} (y) = \wt_{i+2} (y) = 0$ and $\cphi_{i+1}(y) = \wt_{i+1} (y)$. Since $\ce_i(x)=y$, Lemma \ref{lem:ax_s1} implies that $\cepsilon_{i+1}(y) \geq \cepsilon_{i+1}(x)$. Thus, since $\cepsilon_{i+1} (y) = 0$, we have $\cepsilon_{i+1} (x) \leq 0$, which implies, as $\C$ is seminormal, that $\cepsilon_{i+1} (x)=0$. Moreover, $\ce_i(x) = y$ implies that $\wt_{i+2} (x) = \wt_{i+2} (y)$. Thus, we have $\cepsilon_{i+1} (x) = \wt_{i+2} (y) = \wt_{i+2} (x) = 0$ and, consequently, $\qepsilon_{i+1} (x) = 0 = \qepsilon_{i+1}(y)$.

\item Suppose that $\qepsilon_{i+1} (x) \neq \qepsilon_{i+1} (y)$. We will show that $\qepsilon_{i+1} (x) = + \infty$ and $\qepsilon_i (y)=0$. Since $\qepsilon_{i+1} (x) \neq \qepsilon_{i+1} (y)$, we claim that 
\begin{equation}\label{eq:sq_1}
\qepsilon_{i+1} (x) = + \infty, \quad \qepsilon_{i+1} (y) \neq + \infty.
\end{equation}
If $\qepsilon_{i+1} (x) \neq + \infty$ and $\qepsilon_{i+1}(y) = + \infty$, we would have $\cepsilon_{i+1} (x) = \wt_{i+2} (x)$ and $\cepsilon_{i+1} (y) < \wt_{i+2} (y)$. Since $\ce_i(x)=y$, we have $\wt (y) = \wt (x) + \alpha_{i}$, hence $\wt_{i+2} (x) = \wt_{i+2} (y)$. This would imply that 
\[
\cepsilon_{i+1} (y) < \wt_{i+2} (y) = \wt_{i+2} (x) = \cepsilon_{i+1}(x),
\]
which contradicts Lemma \ref{lem:ax_s1}. If $\qepsilon_{i+1} (x), \,\qepsilon_{i+1} (y) \neq + \infty$, then we would have
\[
\cepsilon_{i+1} (y) = \wt_{i+2} (y) = \wt_{i+2} (x) = \cepsilon_{i+1} (x),
\]
and consequently, $\cepsilon_{i+1} (x) = \cepsilon_{i+1} (y)$, which implies $\qepsilon_{i+1} (x) = \qepsilon_{i+1} (y)$, since both are finite, contradicting the hypothesis that $\qepsilon_{i+1} (x) \neq \qepsilon_{i+1} (y)$.
Therefore, equation \eqref{eq:sq_1} holds.
In particular, as $\cepsilon_{i+1} (x) < \wt_{i+2} (x) = \wt_{i+2} (y) = \cepsilon_{i+1} (y)$, by axiom \textbf{S1} we have
\begin{equation}\label{eq:1_eps_1}
\cepsilon_{i+1} (y) = \cepsilon_{i+1} (x)+1.
\end{equation}

Since $\qepsilon_{i+1} (x) = + \infty$, it remains to show that $\qepsilon_i (y) = 0$. 
So, suppose that $\qepsilon_i (y) = k > 0$. In particular, as $\qepsilon_i (y) \neq +\infty$ (because $\qe_i (x) = y$), we have $\cepsilon_{i} (y) = \wt_{i+1} (y) = k$. Consider the connected component consisting of only $i$-labelled edges, containing $y$. Clearly, that component has a unique highest weight element, which is not $y$, since $\cepsilon_i (y) > 0$.
Let $z$ be the highest weight element of that component.
Then, there exists $z_1, \ldots, z_k = z \in \Q$ such that
\begin{equation}\label{eq:1_z}
\begin{cases}
\qe_i (y) = z_1 \;\text{(and thus,}\; \ce_i (y) = z_1),\\
\qe_i (z_l) = z_{l+1} \;\text{(and thus,}\; \ce_i (z_l) = z_{l+1}), \;\text{for}\; l = 1, \ldots, k-1,\\
\qepsilon_i (z_k) = 0 \;\text{(and thus,}\; \cepsilon_i (z_k)=0),
\end{cases}
\end{equation}
as shown in the following diagram:

\begin{center}
\includegraphics[scale=1]{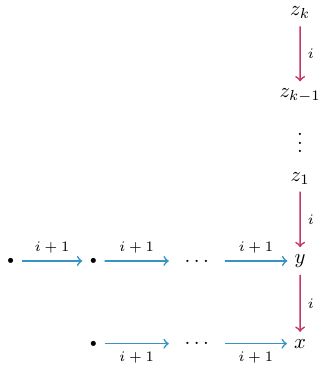}
\end{center}
By \textbf{LQ1}, $\qepsilon_i (z_k) = 0$ implies that $\qphi_{i+1} (z_k)=0$, and thus,
\begin{equation}\label{eq:phi_z}
\cphi_{i+1}(z_k) = 0.
\end{equation}
By \eqref{eq:sq_1}, we have $\qepsilon_{i+1} (y) \neq + \infty$ and hence, 
\begin{equation}\label{eq:phi_i1_k}
\cphi_{i+1} (y) = \wt_{i+1} (y) = k. 
\end{equation}
Thus, there exists $y_1, \ldots, y_k \in \Q$ such that
\begin{equation}\label{eq:1_y}
\begin{cases}
\cf_{i+1} (y) = y_1\\
\cf_{i+1} (y_l) = y_{l+1}, \;\text{for}\; l = 1, \ldots, k-1\\
\cphi_{i+1} (y_k) = 0,
\end{cases}
\end{equation}
as shown in the following diagram:
\begin{center}
\includegraphics[scale=1]{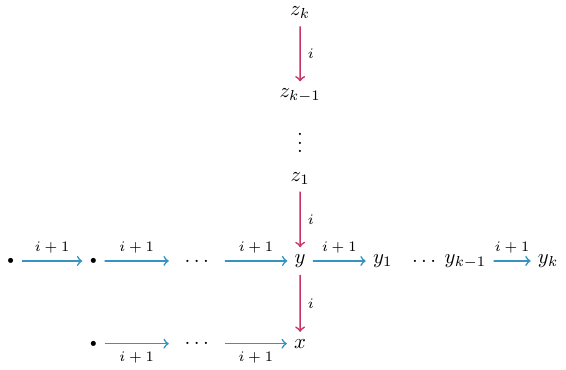}
\end{center}
From equations \eqref{eq:1_eps_1} and \eqref{eq:phi_i1_k}, and Proposition \ref{prop:stem_prop}, we have
\begin{equation}\label{eq:phi_xy}
\cphi_{i+1} (x) = \cphi_{i+1} (y) = k. 
\end{equation}
Then, there exists $x_1, \ldots, x_k \in \Q$ such that
\begin{equation}\label{eq:1_x}
\begin{cases}
\cf_{i+1} (x) = x_1,\\
\cf_{i+1} (x_l) = x_{l+1}, \;\text{for}\; l = 1, \ldots, k-1,\\
\cphi_{i+1} (x_k) = 0.
\end{cases}
\end{equation}
We have $\cf_i(y) = x$ and $\cf_{i+1} (y)= y_1$. From \eqref{eq:1_eps_1} and \eqref{eq:phi_xy}, axiom \textbf{S2$'$} implies that $\cf_i \cf_{i+1} (y) = \cf_{i+1} \cf_i (y)$ and thus
\begin{align*}
\cf_i (y_1) &= \cf_i \cf_{i+1} (y) &\text{by \eqref{eq:1_y}}\\
&= \cf_{i+1} \cf_i(y)\\
&= \cf_{i+1} (x) = x_1 &\text{by \eqref{eq:1_x}}
\end{align*}
Then, we have $\cf_{i} (y_1)=x_1$, or equivalently, $y_1 = \ce_i (x_1)$, as illustrated in the diagram below: 
\begin{center}
\includegraphics[scale=1]{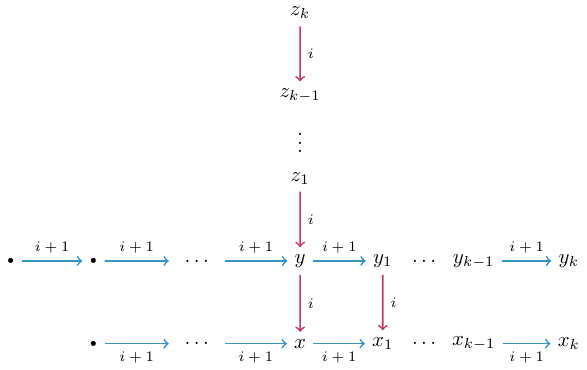}
\end{center}
Since $\cf_{i+1} (x) = x_1$, we have $x = \ce_{i+1} (x_1)$.
By axiom \textbf{S1}, we either have $\cepsilon_{i} (x) = \cepsilon_{i} (x_1)$ or $\cepsilon_{i} (x) = \cepsilon_{i} (x_1)+1$. We will show that both cases lead to contradictions.

Suppose that $\cepsilon_{i} (x) = \cepsilon_{i} (x_1)$. Then, since $\ce_i (x) = y$ and $\ce_i (x_1) = y_1$, \textbf{C1} implies that 
$$k = \cepsilon_{i} (y) = \cepsilon_i (x) - 1 = \cepsilon_i(x_1) -1 =  \cepsilon_{i} (y_1).$$
Therefore, there exists $w_1, \ldots, w_k$ such that
\begin{equation}\label{eq:1_w}
\begin{cases}
\ce_i (y_1) = w_1,\\
\ce_i (w_l) = w_{l+1}, \;\text{for}\; l=1, \ldots, k-1,\\
\cepsilon_i (w_k) = 0.
\end{cases}
\end{equation}
Thus, we have $\ce_{i+1} (y_1) = y$ and $\ce_i (y_1) = w_1$. Since $\cepsilon_i (y) = \cepsilon_i (y_1)$, axiom \textbf{S2} implies that $\ce_i \ce_{i+1} (y_1) = \ce_{i+1} \ce_i (y_1)$ and thus
\begin{align*}
\ce_{i+1} (w_1) &= \ce_{i+1} \ce_i (y_1)  &\text{by \eqref{eq:1_w}}\\
&= \ce_i \ce_{i+1} (y_1)\\
&= \ce_i (y)  &\text{by \eqref{eq:1_y}}\\
&= z_1  &\text{by \eqref{eq:1_z}}
\end{align*}
Applying this reasoning iteratively, we get $\ce_{i+1} (w_l) = z_l$, for $l = 1, \ldots, k$. In particular, we have $\ce_{i+1} (w_k) = z_k$ and thus, $\cf_{i+1} (z_k) = w_k$, as depicted in the following diagram:

\begin{center}
\includegraphics[scale=1]{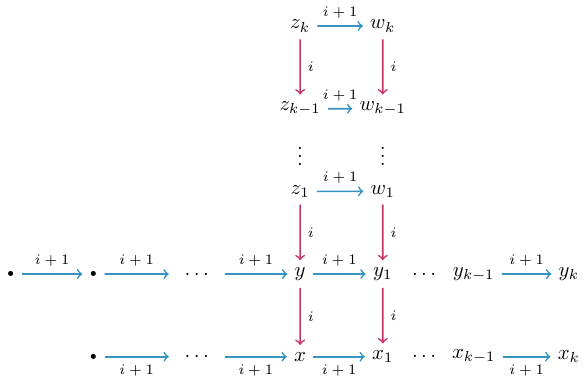}
\end{center} 
This implies that $\cphi_{i+1} (z_k) > 0$, which contradicts \eqref{eq:phi_z}. 

Now suppose that $\cepsilon_i (x) = \cepsilon_i (x_1) +1$.
Since we have $\ce_i (x_1) = y_1$ and $\ce_{i+1} (x_1) = x$, this implies that
\begin{align*}
\cepsilon_{i+1} (y_1) &= \cepsilon_{i+1} (y) + 1 &\text{by  \textbf{C1}}\\
&= (\cepsilon_{i+1} (x) + 1 ) +1 &\text{by \eqref{eq:1_eps_1}}\\
&= \cepsilon_{i+1} (x_1) +1 &\text{by  \eqref{eq:1_x}}
\end{align*} 
Therefore, axiom \textbf{S3} implies that 
$$
\ce_i \ce_{i+1}^{\,2} \ce_i (x_1) = \ce_{i+1} \ce_i^{\,2} \ce_{i+1} (x_1),
$$
and furthermore, 
$$y = \ce_i \ce_{i+1} (x_1) \neq \ce_{i+1} \ce_i (x_1) = y.$$
which is again a contradiction. 

Thus, the original assumption that $\qepsilon_i(y) > 0$ is false, and we have $\qepsilon_i (y) = 0$. \qedhere
\end{enumerate}
\end{proof}
 
\begin{prop}
$\Q$ satisfies axioms \emph{\textbf{LQ3}} and \emph{\textbf{LQ3$'$}}.
\end{prop}

\begin{proof}
We will prove that $\Q$ satisfies axiom \textbf{LQ3}, the proof for \textbf{LQ3$'$} is similar. Let $x \in \Q$ and $i,j \in I$, such that $i \neq j$, and suppose that $\qe_i (x)$ and $\qe_j (x)$ are both defined. This implies that $\ce_i (x)$ and $\ce_j (x)$ are both defined as well and that
\begin{equation}\label{eq:q3_wts}
\cepsilon_i (x) = \wt_{i+1} (x), \quad \cepsilon_j(x) = \wt_{j+1} (x).
\end{equation}

\noindent
\textbf{Case 1.} Suppose that $|i-j|>1$. Then, axiom \textbf{S1} implies that $\cepsilon_i (\ce_j (x)) = \cepsilon_i (x)$, and therefore, by axiom \textbf{S2},
\begin{equation}\label{eq:q3_ij}
\ce_i \ce_j (x) = \ce_j \ce_i (x).
\end{equation}
Since $|i-j|>1$, we have $\wt_{i+1} (x) = \wt_{i+1} (\ce_j (x))$. Thus, \eqref{eq:q3_wts} implies that
$$\cepsilon_i (\ce_j (x)) = \cepsilon_i (x) = \wt_{i+1} (x) = \wt_{i+1} (\ce_j (x)),$$
and consequently, $\qepsilon_i (\ce_j (x)) = \qepsilon_i (\qe_j (x)) \neq + \infty$. Thus, $\qe_i \qe_j (x) = \ce_i \ce_j (x)$. Applying the same reasoning, we get $\qe_j \qe_i (x) = \ce_j \ce_i (x)$. Therefore, by \eqref{eq:q3_ij}, we have $\qe_i \qe_j (x) = \qe_j \qe_i (x)$.

\noindent
\textbf{Case 2.} Suppose that $|i-j|=1$, and without loss of generality, suppose that 
$$\qe_i (x) = \ce_i (x) = y, \quad \qe_{i+1} (x) = \ce_{i+1} (x) = z.$$
Since $\qe_{i+1} (x)$ is defined, we have $\qepsilon_{i+1} (x) \neq + \infty$. Therefore, as $\qe_i (x) =y$, Proposition \ref{prop:q_LQ2} implies that
\begin{equation}\label{eq:q3_i1_xy}
\qepsilon_{i+1} (x) = \qepsilon_{i+1} (y) \neq +\infty,
\end{equation}
and, consequently, $\cepsilon_{i+1} (x) = \cepsilon_{i+1} (y)$. Thus, it follows from axiom \textbf{S2} that
\begin{equation}\label{eq:q3_i1_com}
\ce_i \ce_{i+1} (x) = \ce_{i+1} \ce_i (x).
\end{equation}
From \eqref{eq:q3_i1_xy} and \eqref{eq:q3_i1_com}, we have $\qe_{i+1} \qe_i (x) = \ce_{i+1} \ce_i (x) = \ce_i \ce_{i+1} (x)$. Thus, it remains to show that $\qe_i \qe_{i+1} (x) = \ce_i \ce_{i+1} (x)$. We claim that $\qphi_i (z) \neq + \infty$. If $\qphi_i (z) = + \infty$ and $\qphi_{i+1} (x) > 0$, Propositions \ref{prop:local_ax} (3.2) and \ref{prop:q_LQ2} would imply that $\qphi_{i} (x) = + \infty$, which contradicts $\qe_i (x)$ being defined. If $\qphi_i (z) = + \infty$ and $\qphi_{i+1} (x) = 0$, then Proposition \ref{prop:q_LQ1} would imply that $\qepsilon_i (x) = 0$, which also contradicts $\qe_i (x)$ being defined. Therefore, we have $\qphi_i (z) \neq + \infty$, and thus $\qepsilon_i (z) \neq +\infty$. Therefore, $\qe_i \qe_{i+1} (x) = \ce_i \ce_{i+1} (x)$.
\end{proof}

\printbibliography
\end{document}